\documentclass{article}

\usepackage{arxiv}

\usepackage[utf8]{inputenc} 
\usepackage[T1]{fontenc}    
\usepackage{hyperref}       
\usepackage{url}            
\usepackage{booktabs}       
\usepackage{amsfonts}       
\usepackage{nicefrac}       
\usepackage{microtype}      
\usepackage{lipsum}		
\usepackage{graphicx}
\usepackage{natbib}
\usepackage{doi}

\usepackage{amsmath}
\usepackage{mathtools}
\usepackage{amsthm}

\usepackage{esint}
\usepackage{enumerate}
\usepackage{amssymb}
\usepackage{braket}
\usepackage{bm}
\usepackage{graphicx}
\usepackage{caption}
\usepackage{subcaption}
\usepackage{xcolor}
\usepackage{ulem}
\numberwithin{equation}{section}
 \usepackage{tikz}
\usepackage{natbib}
\usepackage{dsfont}
\usepackage{comment}
\usepackage{float}

\usepackage{thmtools, thm-restate}

\usepackage{algorithm}
\usepackage{algpseudocode}
\usepackage{algorithmicx}

\usepackage{hyperref}

\providecommand{\U}[1]{\protect\rule{.1in}{.1in}}
\newtheorem{theorem}{Theorem}[section]

\newtheorem{lemma}[theorem]{Lemma}

\newtheorem{proposition}[theorem]{Proposition}
\newtheorem{remark}[theorem]{Remark}

\newtheorem{assumption}[theorem]{Assumption}

\global\long\def\defeq{\stackrel{\textrm{def}}{=}}%

\newcommand{\dx}{\, \dd x}

\newcommand{\hd}{\hat{d}}
\newcommand{\hb}{\hat{b}}
\newcommand{\B}{\mathcal{B}}
\newcommand{\dy}{\, \dd y}

\newcommand{\vol}{\mathrm{vol}}

\newcommand{\T}{T}
\newcommand{\mo}{k}
\newcommand{\mS}{\mathcal{S}}
\newcommand{\prob}{
\mathrm{Prob}
}
\newcommand{\C}{\mathcal{C}}

\newcommand{\diameter}{\mathrm{Diameter}}

\newcommand{\dist}{\mathrm{dist}}

\newcommand{\OO}{\mathcal{O}}
\newcommand{\R}{\mathbb{R}}

\newcommand{\M}{\mathcal{M}}

\definecolor{mygreen}{rgb}{0.1,0.75,0.2}

\newcommand{\D}{\mathbf{D}}

\newcommand{\E}{\mathbb{E}}

\newcommand{\mbe}{\hat{b}}

\newcommand{\Yn}{Y_{1:n}}

\newcommand{\eps}{\varepsilon}
\newcommand{\wideeps}{\widetilde{\varepsilon}}
\newcommand{\widedelta}{\widetilde{\delta}}

\renewcommand{\P}{\mathbb{P}}

\newcommand{\chenghui}[1]{\textcolor{blue}{CL: #1}}

\newcommand{\dd}{\mathrm{d}}

\definecolor{darkred}{rgb}{0.6,0.1,0.1}
\definecolor{darkgreen}{rgb}{0.1,0.6,0.1}
\definecolor{darkblue}{rgb}{0.1,0.1,0.6}

\def\rm#1{\textbf{\textcolor{darkgreen}{#1}}}

\newcommand\smallO{
  \mathchoice
    {{\scriptstyle\mathcal{O}}}
    {{\scriptstyle\mathcal{O}}}
    {{\scriptscriptstyle\mathcal{O}}}
    {\scalebox{.7}{$\scriptscriptstyle\mathcal{O}$}}
  }

\newcommand{\jessi}[1]{{\color{cyan}[[Jessi: #1]]}}
\newcommand{\True}[1]{{\mathbf{#1}_{\text{True}}}}
\newcommand{\true}[1]{{{#1}_{\text{True}}}}

\usepackage{comment}

\title{A Divide-and-Conquer Approach to Persistent Homology}

\author{Chenghui Li\thanks{
    The authors gratefully acknowledge support from NSF under Grant Number DMS 2038556.  JCK also gratefully acknowledges support from NSF under Grant Number DMS 2337243.
    This research was performed using the computing resources and assistance of the UW-Madison Center For High Throughput Computing (CHTC) in the Department of Computer Sciences \citep{chtc}. The CHTC is supported by UW-Madison, the Advanced Computing Initiative, the
Wisconsin Alumni Research Foundation, the Wisconsin Institutes for Discovery, and the National Science Foundation, and is an active member of the Open Science Grid, which is supported by the National Science Foundation and the U.S. Department of Energy’s Office of Science.}\hspace{.2cm}
    and 
    Jessi Cisewski-Kehe\\
    Department of Statistics, University of Wisconsin-Madison}

\renewcommand{\headeright}{}
\renewcommand{\undertitle}{}
\renewcommand{\shorttitle}{Divide-and-conquer persistent homology}

\hypersetup{
pdftitle={A Divide-and-Conquer Approach to Persistent Homology},
pdfauthor={Chenghui Li, Jessi Cisewski-Kehe},
pdfkeywords={Big data, computational statistics, geometric data analysis, shape analysis, topological data analysis},
}

\begin{document}
\maketitle

\begin{abstract}
	Persistent homology is a tool of topological data analysis that has been used in a variety of settings to characterize different dimensional holes in data.  However, persistent homology computations can be memory intensive with a computational complexity that does not scale well as the data size becomes large. In this work, we propose a divide-and-conquer (DaC) method to mitigate these issues. The proposed algorithm efficiently finds small, medium, and large-scale holes by partitioning data into sub-regions and uses a Vietoris-Rips filtration. Furthermore, we provide theoretical results that quantify the bottleneck distance between DaC and the true persistence diagram and the recovery probability of holes in the data. We empirically verify that the rate coincides with our theoretical rate, and find that the memory and computational complexity of DaC outperforms an alternative method that relies on a clustering preprocessing step to reduce the memory and computational complexity of the persistent homology computations. Finally, we test our algorithm using spatial data of the locations of lakes in Wisconsin, where the classical persistent homology is computationally infeasible.
\end{abstract}


\section{Introduction}\label{sec:intro}
Topological data analysis (TDA) provides a framework for quantifying aspects of the shape of data.
Persistent homology is a popular TDA tool that computes a multi-scale summary of the different dimensional holes in a dataset.  
The summary information is contained in a \textit{persistence diagram}, which can be used for visualizing complex data, but also may be used directly or transformed in statistical or machine learning models for inference or prediction tasks \citep[e.g.][]{bubenik2015statistical, reininghaus2015stable,Anirudh2016,adams2017persistence,robinson2017hypothesis,kusano2017kernel,biscio2019accumulated,berry2020functional,hensel2021survey}.  
Many disciplines have benefited from these techniques, such as astronomy \citep{cisewski2014non, kimura2017quantification, Pranav:2017vy, Green:2019uz, Pranav:2019ws, Xu_2019, Cole:2020ul, wilding2021persistent, cisewski2022differentiating}, image analysis \citep{Qaiser, cisewski2023weighted, glenn2024confidence}, medicine and health \citep{bendich2014persistent, Lawson:2019tg}, material science \citep{kramar}, time series analysis \citep{khasawneh2016chatter}, among others.

A persistence diagram reveals the birth and death times of homology group generators, which we refer to as ``features'', in a particular filtration.  
 Zero-dimensional homology group generators ($H_0$ features) may be thought of as connected components or clusters, one-dimensional homology group generators ($H_1$ features) as closed loops, two-dimensional homology group generators ($H_2$ features) as voids in $\R^3$ like the interior of a soccer ball or balloon, and $p$-dimensional homology group generators ($H_p$ features) as $p$-dimensional holes.  
Some features can be extensive and traverse large portions of the data domain, while others can be small or medium-sized and confined to a localized region.  Depending on a particular application, all feature sizes may be scientifically informative. 

When pursuing a persistent homology analysis of a dataset, the computation of a persistence diagram is a necessary initial step; however, the computation of a diagram is computationally intensive and does not scale well with increasing data size. In the worst case, for a sample size of $n$, computation of a persistence diagram using the standard algorithm scales as $2^{3n}$ \citep{otter2017roadmap}.\footnote{Note that this scaling is based on the worst-case number of simplexes for a Vietoris-Rips complex (defined in Section~\ref{sec:background}) as $2^{n}-1$ and considering simplexes up to a dimension of $n-1$, along with the standard algorithm for computing a persistence diagram which scales as $(2^{n}-1)^3$ \citep{otter2017roadmap}.  However, when computing a persistence diagram, the dimension of simplexes is often limited based on the dimension of the data, resulting in fewer simplexes.}
This computational limitation is problematic because even moderately-sized datasets can be prohibitive (e.g., for 3D data, an $n$ of even just a few thousand can be problematic, \citealt{malott2021distributed}).  For example, in cosmology, where massive cosmological simulations of the large-scale structure of the Universe are a primary source of analysis, the simulated data
can have hundreds of millions to billions of simulation particles that may be converted into catalogs of hundreds of thousands or more dark matter halos and subhalos
(e.g., \citealt{springel2005cosmological, hellwing2016copernicus, libeskind2018tracing, liu2018massivenus,villaescusa2020quijote}).
Even if a single persistence diagram can be computed, inference may require bootstrapping and the computation of thousands of persistence diagrams \citep{Fasy_2014, Xu_2019, berry2020functional}.

New computational algorithms have been developed to mitigate these challenges (e.g., \citealt{Sheehy2012Linear, bauer2014distributed, otter2017roadmap, bauer2021ripser}). By defining a different filtration, \cite{Sheehy2012Linear} provides a theoretically linear computational complexity algorithm to \textit{approximate} the Vietroris-Rips (VR) filtration; see Section~\ref{sec:background} for an overview of the VR filtration. However, their algorithm only provides a constant factor difference guarantee between the estimated and true quantities on the persistence diagram (i.e., the birth and death times of the features). Our proposed method targets a more precise estimate of the birth and death times of the features with a $\delta$-approximation. \citealt{bauer2014distributed}, \citealt{otter2017roadmap} and \citealt{bauer2021ripser} consider other approaches to improve the efficiency. For example, \cite{bauer2014distributed} consider a distributed method in the persistent homology computation (i.e., the reduction of the boundary matrix) after the formation of the filtered simplicial complexes.
Other approaches focus on reducing the size of the data through the 
specification of landmarks or a pre-processing clustering step (e.g., \citealt{gamble2010exploring,moitra2018cluster,verma2021data}).
These approaches can smooth out or exclude potentially scientifically interesting small- or medium-scale homological features.

\cite{malott2019fast} takes a hybrid approach where large-scale features are detected using a pre-processed clustering dataset, and small-scale features are detected by partitioning the data.  This method is referred to as \textit{Partitioned Persistent Homology} (PPH) and has a distributed version presented in \cite{malott2021distributed}.
Two limitations of the PPH approach are that the large-scale features are approximated by a sub-sampling process that is affected by the density of the sample, and medium-scale features that traverse multiple partitions may be missed.  Additionally, there is no statistical evaluation of the approximation.  Our proposed method considers a similar data partitioning but is designed to improve on these limitations to identify small-, medium-, and large-scale features and provide probabilistic guarantees.
In particular, we propose a divide-and-conquer (DaC) method to estimate a persistence diagram for a large dataset.  Like PPH, the proposed DaC approach helps to relieve both the computational complexity and memory burden and is readily parallelized or distributed across multiple machines under this scheme. 
Unlike PPH, we develop estimators of features that may cross into multiple partitions and do not rely on clustering to uncover large-scale features.
While partitioning reduces the memory requirement, there is a cost associated with the accuracy of the resulting diagram.
We quantify the estimation error of the proposed DaC method and 
show that the true persistence diagram can be recovered with high probability.

The rest of this article is organized as follows. In Section~\ref{sec:background}, background on persistent homology is provided. In Section~\ref{sec:methodology}, the DaC method is introduced along with an illustrative example. We also establish a theoretical framework to study the recovery probability of DaC. In Section~\ref{sec:simulation}, the results of a simulation study to verify the theoretical results and compare DaC with a cluster-based algorithm are presented.  A real-data example using the locations of lakes in Wisconsin is presented in Section~\ref{sec:app}.  Discussion and concluding remarks are in Section~\ref{sec:conclusion}.



\section{Background} \label{sec:background}
The mathematical field of homology counts the number of different dimensional holes in a topological space, such as a manifold, which can be used to classify the space \citep{munkres2000topology,munkres2018elements}.  Persistent homology extends the notion of counting holes in a general topological space to counting holes in data. 
For point-cloud data, there are typically only $H_0$ features since the points are not connected.  
To investigate the persistent homology of such a dataset, we need an intermediate structure to quantify the connectedness of the points across a filtration. Several options exist, including different types of abstract simplicial complexes (e.g., alpha complexes, \v{C}ech complexes, Vietoris-Rips complexes). 
This work uses Vietoris-Rips (VR) complexes, though the proposed method may be generalized to other complexes.

In $\R^3$, a simplicial complex, $\mathcal S$, may be composed of a finite set of zero-simplexes (points), one-simplexes (segments), two-simplexes (triangles), and three-simplexes (tetrahedra); a geometric $k$-\textit{simplex} is the convex hull of $k+1$ affinely independent points. Denote the collection of $k$-simplexes as $\mS_k$ with $\mS=\cup_{k=0}^\infty \mS_k$, and the number of simplexes in $\mS_k$ as $n_k$. If $\sigma_1$ and $\sigma_2$ are simplexes such that $\sigma_1 \subset \sigma_2$, then $\sigma_1$ is called a \textit{face} of $\sigma_2$.
If for each simplex $\sigma \in \mathcal S$, every face of $\sigma$ is also in $\mathcal S$, and if $\sigma_1, \sigma_2 \in \mathcal S$ is such that either $\sigma_1 \cap \sigma_2 \in \mathcal S$ or $\sigma_1 \cap \sigma_2 = \emptyset$, then $\mathcal S$ is a {\it simplicial complex}.

Given observations $\Yn = \{y_1, y_2, \ldots, y_n\}\subset \R^D$, to compute the homology of $\Yn$, a VR complex may be constructed by picking a non-negative distance scale $t \in \R_{\geq 0}$, with $\Yn$ as the vertex set (zero-simplexes).  Any two points $y_i, y_j \in \Yn$ with Euclidean distance less than or equal to $t$, $\rho(y_i,y_j)\leq t$, are connected by a segment (one-simplex). 
Similarly, triplets of points with pairwise distances less than or equal to $t$ form a triangle (two-simplex), and quadruples of points with pairwise distances less than or equal to $t$ form a tetrahedron (three-simplex).  
In other words, $\Yn$ are the center of balls with diameter $t$, and balls that intersect can form segments, triangles, and tetrahedrons, etc.
The VR complex at scale $t$ for vertex set $\Yn$ can be defined as \citep{edelsbrunner2010computational}
\begin{equation}
\text{VR}_t(\Yn) = \{\sigma \subseteq \Yn \mid \rho(u,v) \leq t, \forall u, v \in \sigma\}.  \label{eq:vr}
\end{equation}
The set of $k$-simplexes of $\text{VR}_t(\Yn)$ is denoted as $\text{VR}_t^k(\Yn)$, with $\text{VR}_t(\Yn)=\cup_{k=0}^\infty \text{VR}^k_t(\Yn)$.  Illustrations of these concepts are displayed in Figures~\ref{fig:circle1} and \ref{fig:circle2} for 2D data sampled on four circles.  The data are the black points with cyan balls of radius 0.4 and 0.7, respectively; the one-simplexes are black segments and the two-simplexes are the yellow triangles.
	\begin{figure}[htbp]
		\par\medskip
		\centering
		\begin{subfigure}[t]{0.31\textwidth}
			\centering
    	        \includegraphics[width=\textwidth]{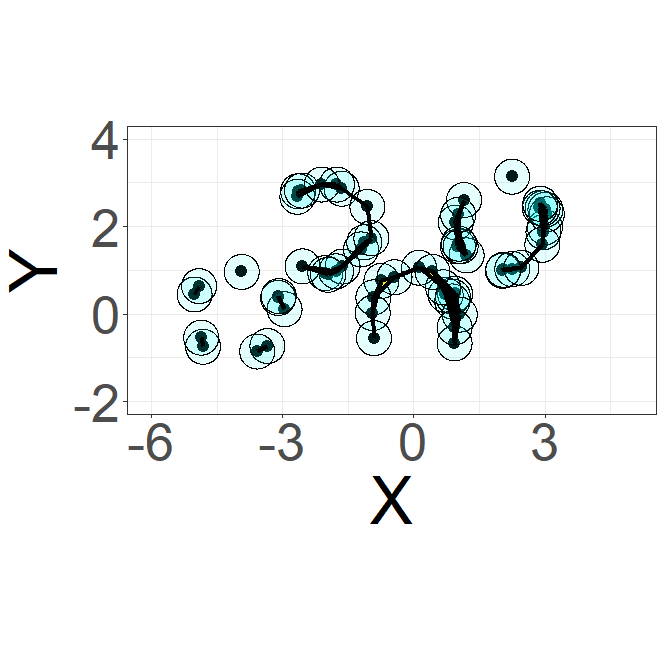}
			\caption{$t=0.8$}
			\label{fig:circle1}
			\end{subfigure}
		\begin{subfigure}[t]{0.31\textwidth}
			\centering
			\includegraphics[width=\textwidth]{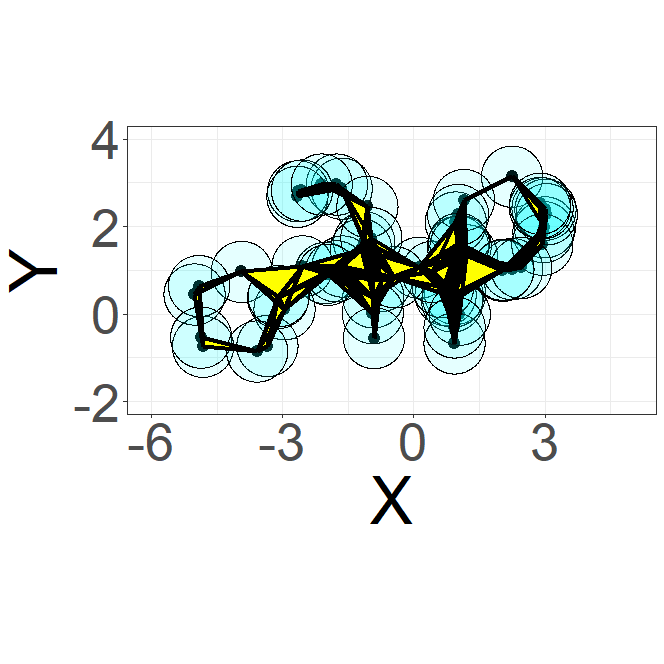}
			\caption{$t=1.4$}
			\label{fig:circle2}
		\end{subfigure}
		\begin{subfigure}[t]{0.31\textwidth}
			\centering
			\includegraphics[width=\textwidth]{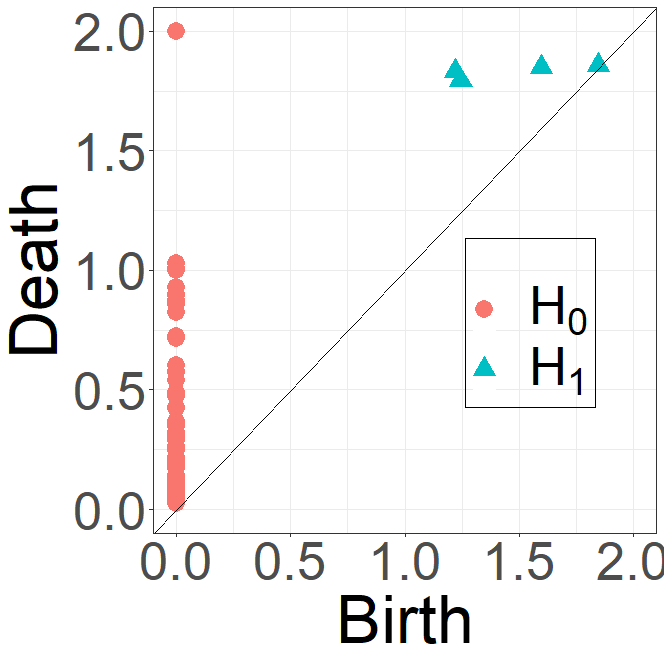}
			\caption{Persistence Diagram}
			\label{fig:circle_diagram}
		\end{subfigure}
		\caption{Illustration of VR complexes and the persistent diagram of 60 data points sampled from four circles with noise. The data in (a) and (b) are displayed as black points with cyan balls with radiuses of $0.8/2$ and $1.4/2$, respectively, along with one-simplexes (black segments) and two-simplexes (yellow triangles). The corresponding persistence diagram (c) has four $H_1$ features (blue triangles) and 60 $H_0$ features (red discs).}
	\end{figure}

Persistent homology considers a range of values for $t$, and constructs a VR filtration such that for $t_1 < t_2$, $\text{VR}_{t_1}(\Yn) \subseteq \text{VR}_{t_2}(\Yn)$.  
The filtration is controlled by the distance $t$, which starts at 0 and increases to $\infty$, at which the corresponding homology of $\text{VR}_t$ complex is computed.  
At $t=0$, the homology is computed over the vertex set and generally has only $H_0$ features (i.e., no higher dimensional holes).  
As $t$ increases, $H_0$ features merge, and higher dimensional holes may appear or disappear.  
The time $t$ when a feature appears and disappears from the filtration is called its \textit{birth} and \textit{death} time, respectively.
The \textit{persistence} or the lifetime of a feature is the difference between its death and birth times.
For example, when $t$ increases from $0.8$ in Figure \ref{fig:circle1} to $1.4$ in Figure \ref{fig:circle2}, connected components merge, and loops form. 
This information is summarized in the \textit{persistence diagram} displayed in Figure~\ref{fig:circle_diagram}.  
A persistence diagram, $\D$, can be defined as 
\begin{equation}
    \D = \{(r_j, b_j, d_j): j = 1, \ldots, |\D|\} \cup \Delta,
\end{equation}
where $r_j$ represents the homology group dimension of feature $j$ with birth time $b_j$ and death time $d_j$, the $|\D|$ gives the total number of homology group generators detected with $d_j > b_j$, and $\Delta$ signifies features with zero persistence (birth time = death time) which are considered to have infinite multiplicity \citep{mileyko2011probability}.  

It is common to interpret features on a persistence diagram with higher persistence (i.e., further from the diagonal) as holes that are real topological signals, while those with low persistence (i.e., close to the diagonal) as topological noise \citep{Fasy_2014} or revealing some geometric curvature information \citep{bubenik2020persistent,turkes2022onPH}.   
For a more detailed introduction to persistent homology, see \cite{edelsbrunner2008persistent,otter2017roadmap}. Unless otherwise noted, persistence diagram computations presented in this paper use Dionysus \citep{morozov2007dionysus} via the {\tt R} package {\tt TDA} \citep{fasy2021package}.

\subsection{Persistent Homology Algorithm}
Different dimensional holes are detected in a VR filtration using notions from algebraic topology.  A boundary matrix is used to capture the information about the simplexes that form across the filtration.  Let $B_k$ represent the $n_{k-1} \times n_k$ boundary matrix of the $k$-simplexes that form in the filtration, where the rows represent the $k-1$-simplexes that form the boundaries of the $k$-simplexes:  $B_k(i,j)$ is 1 if $k-1$-simplex $i$ is a boundary of $k$-simplex $j$, otherwise it is 0.  
The \textit{Standard Algorithm} (SA) for computing a persistence diagram reduces the boundary matrices in a particular way so that the reduced matrices, $R_k$, reveal the birth and death times of the different dimensional features \citep{edelsbrunner2000topological,zomorodian2005computing}.  The $H_k$ features are the non-zero columns of $R_k$, which can be used as (non-unique) \textit{representations} of those features.  
For example, if column $j$ of $R_1$ has ones in rows $i_1, i_2, i_3, i_4$, then these zero-simplexes may be used to represent $H_1$ feature $j$.

In the worst case, the SA has cubic complexity in the number of simplexes. This bound is sharp \citep{morozov2005persistence}, but when the boundary matrix is sparse, the complexity is less than cubic. There are other algorithms for the reduction of the boundary matrix introduced in \cite{milosavljevic2011zigzag} taking $\OO(n_S^{2.376})$, where $n_S$ is the total number of simplexes. 
We use the SA in this paper without any speed-up in the reduction step. 
More details of these computations are beyond the scope of the proposed work, but an accessible presentation is available in \cite{otter2017roadmap}.

Though persistent homology is useful for summarizing topological information, common approaches like the VR filtration suffer both in computational efficiency and the required memory. Constructing $k$-simplices ($\mathcal{S}_k$) takes $O(n^{k+1})$ operations for $n$ points because every $(k+1)$-subset is taken into account to construct the $k$-simplex; this is a substantial burden when considering large datasets (such as cosmological simulations which can have millions or billions of data points). An algorithm is needed to simplify the boundary matrix, which suffers from memory and computational issues.
Unfortunately, the VR filtration computation is challenging to parallelize or distribute across multiple machines because the global structure of the data is needed to determine when every feature is born or dies.

\subsection{Approximations of Persistent Homology}
Various methods have been proposed to address the computational burden of computing a persistence diagram. 
A straightforward framework considers reducing the sample size through a subsampling method to achieve sparse representation. This method takes $k_c$ data points from the initial point cloud and runs the persistent homology on them, where $k_c \ll n$. For example, \cite{chazal15} considers subsampling and merging average persistence landscapes (which are functional summaries of persistence diagrams, \citealt{bubenik2015statistical}) as a functional summary of a partitioned persistence diagram. \cite{moitra2018cluster} utilizes k-means++ to speed up the computation, but this can lead to missing small homological features. \cite{malott2019fast} proposes using k-means++ with an added zoom-in step to recover small features. Overall, the subsampling methods can help reduce computational complexity and memory while preserving birth and death time estimates. However, the success of these methods heavily relies on the quality of the distance-based clustering, which makes it difficult to consider a rigorous theoretical analysis. Choosing the right number of reduced data points $k_c$ can also be tricky. In general, cluster-based algorithms lack the ability to detect local topological information. For instance, clustering techniques require an approximately equal density of observations on each feature to recover all the features in the data space, whereas TDA methods are generally not bound by this.  The k-means++ approach of \cite{moitra2018cluster} is empirically compared to the proposed DaC method in Section~\ref{sec:simulation}.



\section{Methodology}\label{sec:methodology}

To mitigate the memory and computational burden when computing persistence diagrams using VR filtration, we propose a DaC method. The proposed DaC approach begins by dividing the data into sub-regions with artificial boundaries added to complete holes that may traverse into other sub-regions. A persistence diagram is computed for each sub-region, and these persistence diagrams are merged to estimate a persistence diagram for the full data volume. The challenge is to recover holes that span more than one sub-region; see further discussion in Section~\ref{sec:general framework}. 


The estimated persistence diagram using the proposed DaC method can recover the features of the \textit{true diagram}, which we define as the persistence diagram computed using the entire ``non-partitioned'' dataset. In particular, with high probability, the DaC method can detect the features from the true diagram, with statistical guarantees on the proximity of the estimated birth and death times to their true values. 

In the following subsections, we begin with a simple example to illustrate the main concepts of the proposed DaC algorithm, followed by an overview of the general framework and main theoretical results.  The general technical results and proofs are available in the Appendix. 

\subsection{Illustrative example of the DaC method}
The proposed DaC algorithm is introduced in this section considering a simple 2D dataset sampled on one large loop and two smaller loops with noise as displayed in Figure \ref{fig:ExampleData}.  The goal is to recover the three loops on the DaC-estimated persistence diagram, which can be compared to the true persistence diagram that uses all the data points.  Since the dataset is small, the true persistence diagram of the $H_1$ features can be computed; representative loops of the true diagram are displayed in Figure~\ref{fig:ExampleFullLoops}.

To illustrate the DaC method, we divide the data domain into two sub-regions splitting through the middle of the x-axis, as displayed in Figure~\ref{fig:ExampleLoops} by the vertical green dashed line. If we simply compute the VR filtration on the sub-data, then the loops away from the boundary can be easily recovered. However, the largest circle across the shared boundary of the sub-regions may not be recovered without additional steps.  To recover the split  feature, a \textit{supplemental boundary} is added along the vertical green dashed line.  This supplemental boundary behaves like additional data in the VR filtration, where distances are computed between it and the original data points. 

Representations of the loops recovered from computing a VR filtration on the data from two sub-regions are displayed in different colors in  Figure~\ref{fig:ExampleLoops}. The next step is to merge any features that cross the supplemental boundary; we call a sub-feature in a sub-region that forms with the supplemental boundary a \textit{potential-feature}, and a feature fully contained in a sub-region a \textit{complete-feature}.
A \textit{topological projection} (i.e., the portion of boundary with representative data points forming a feature) onto the supplemental boundary can be used to merge potential-features. If potential-features in different sub-regions share the same projection, they may be divided from one larger feature. If there are more than two sub-regions, a similar method can be used to merge them, or a distance matrix can be constructed among all of the potential-features to be used in a VR filtration to identify which potential-features can be merged; see the discussion in Section~\ref{sec:method:combination method} for details. As shown in the persistence diagram displayed in Figure \ref{fig:ExamplePD}, the merged feature's birth and death times (green triangle) match the true feature's birth and death times (blue square).
	\begin{figure}[htbp]
		\par\medskip
		\centering
		\begin{subfigure}[t]{0.45\textwidth}
			\centering    	        \includegraphics[width=5.5cm]{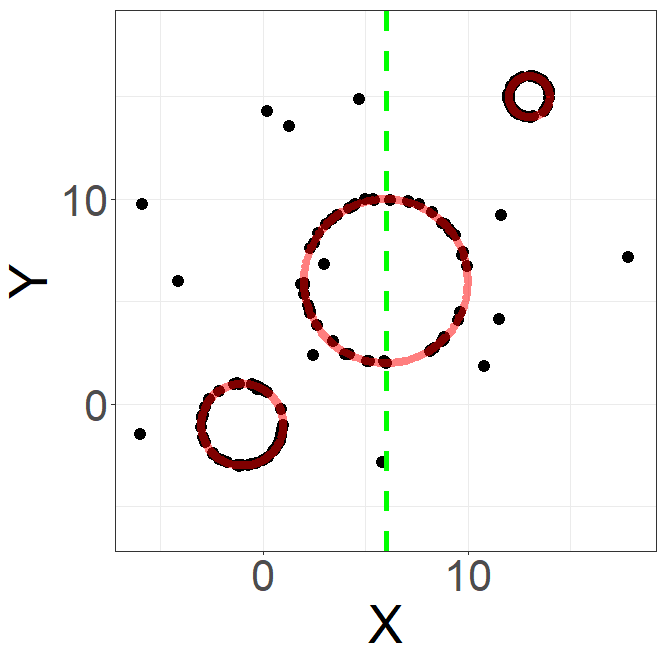}
			\caption{Data}
			\label{fig:ExampleData}
		\end{subfigure}
		\begin{subfigure}[t]{0.45\textwidth}
			\centering
			\includegraphics[width=5.5cm]{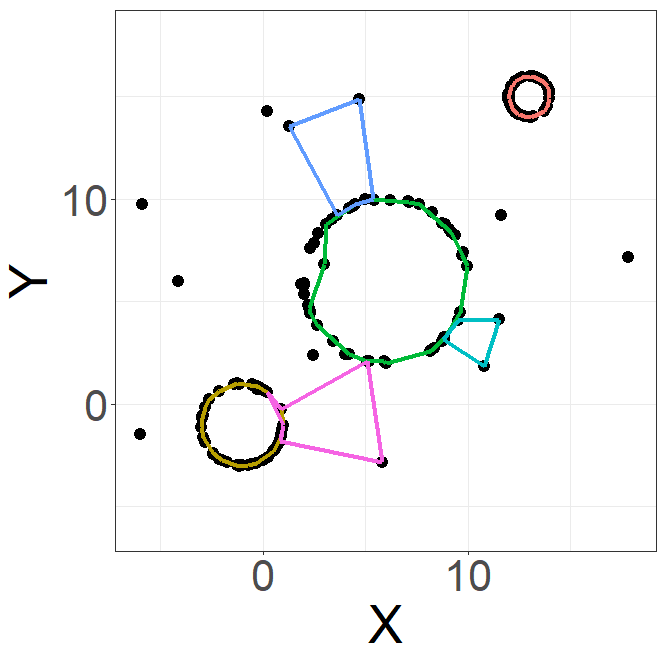}
			\caption{$H_1$ features recovered from full data}
			\label{fig:ExampleFullLoops}
		\end{subfigure}
		\begin{subfigure}[t]{0.45\textwidth}
			\centering
			\includegraphics[width=5.5cm]{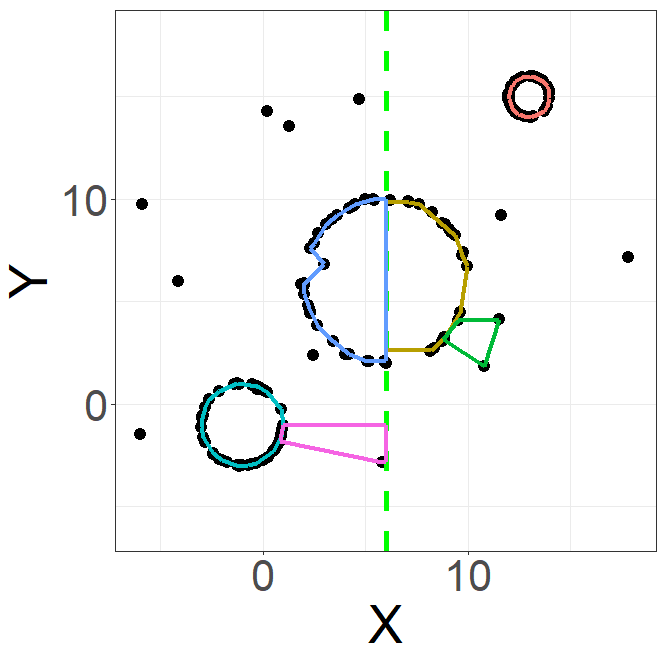}
			\caption{$H_1$ features recovered from DaC}
			\label{fig:ExampleLoops}
		\end{subfigure}
		\begin{subfigure}[t]{0.45\textwidth}
			\centering
			\includegraphics[width=5.5cm]{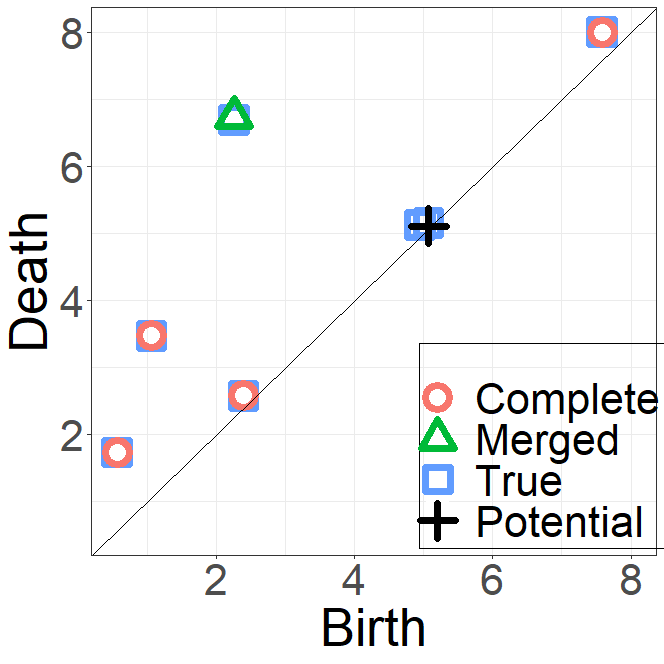}
			\caption{$H_1$ Persistence Diagrams}
			\label{fig:ExamplePD}
		\end{subfigure}
		\caption{Example of DaC method.  (a) The data (black points) are sampled on three circles (red lines) with noise. Representations of the loops from the true (b) and the sub-region DaC (c) persistence diagrams. (d) The true (blue squares) and DaC $H_1$ persistence diagrams, with the type of feature (complete, merged, true, potential) indicated by the color and symbol.}\label{fig:SimpleExample}
	\end{figure}

\subsection{General framework}\label{sec:general framework}
Given data points $Y_{1:n}=\{y_i\}_{i=1}^n$ sampled in $\R^D$, and the true persistence diagram (based on the full dataset) denoted as 
\begin{equation}
\D^0 = \{(r_{j}^0, b_{j}^0, d_{j}^0): j = 1, \ldots, |\D^0|\} \cup \Delta,
\end{equation}
we discuss the steps to estimate $\D^0$ using the proposed DaC method. 
To simplify the presention, let $[n]=1,2,\ldots,n$, and $\rho(y_i,y_j)$ be the Euclidean distance between $y_i$ and $y_j$. 

\subsubsection{Sub-region distance matrix construction}\label{sec:method:distance matrix construction}

Assume the data are divided into $m$ sub-regions. To recover the individual sub-region diagrams, the sub-region partitions are used to create shared boundaries, $\B$. Then, a VR filtration is computed on the distance matrix within each sub-region that includes distances between the data points in the sub-region, and distances between the data points and the supplemental boundaries of their sub-region. The distances between data points is a Euclidean distance, and the distance between a data point $x$ and the supplemental boundary $\B$ is defined as 
\begin{equation}
    \dist(x,\B)\defeq \min_{y\in\B} \rho(x,y).
\end{equation}
 We call a detected feature in a sub-region a \textit{sub-feature}. We call a sub-feature that does not form with the supplemental boundary a \textit{complete-feature}, and its associated diagram a \textit{complete diagram}. If a sub-feature forms with a supplemental boundary, then we call its associated diagram a \textit{potential diagram}. Notice that the number of supplemental boundaries added for each sub-region is at most $2D$ for hyperrectangle partitions, which does not change the order of the computation and memory burden for computing sub-features. 

The resulting \textit{sub-region-$l$ potential diagram} is denoted as
\begin{equation}\label{eq-def:sub-regtion-k diagram}
\D^l = \{(r_{j}^l, b_{j}^l, d_{j}^l): j = 1, \ldots, m_l\} \cup \Delta, l=1, \ldots, m,
\end{equation}
where $m_l$ is the number of potential-features within the $l$th sub-region. Suppose feature $(r_{j}^l, b_{j}^l, d_{j}^l)$ is associated with a subset of the observations, denoted $\{y^l_j(i)\}_{i\in[n_{j}^l]}$, where $n^l_j$ is the number of representative points for potential-feature $(r_{j}^l, b_{j}^l, d_{j}^l)$.

\subsubsection{Potential-Feature Merge Methodology}\label{sec:method:combination method}

Features fully contained in a sub-region are easily detected with the persistence diagram computation of the data points in each sub-region.  The primary challenge is detecting features that cross the supplemental boundary because these so-called potential-features are divided into two or more partitions.  Two approaches for merging potential-features are proposed, the Projected Merge method and the Representative Rips Merge method, where the former is recommended when $m$ is small, and the latter when $m$ is large.

\paragraph{Projected Merge method}
Let $\B^l_j$ denote the supplemental boundaries for potential-feature $j$ in sub-region $l$. 
We denote 
$\{y^l_j(i)\}_{i\in[n_{j}^l]}\in\M_{j}^{(l)}\subseteq \M$ where $\M_{j}^{(l)}$ is the part of the true feature $\M$ that is contained in $l$th sub-region. 
Define the \textit{topological projections} $\mathcal{P}^l_j$ of $\{y^l_{j}(i)\}_{i\in[n^l_j]}$ as $\mathcal{P}^l_j\subseteq\B^l_j$ such that $\mathcal{P}^l_j\cup \M_{j}^{(l)}$ forms a $r^l_j$-dimensional feature. 

To merge the potential-features that may contribute to the same true feature, the topological projections help by determining which potential-features from different sub-regions are projected onto the same boundary and therefore may be combined.
The mathematical construction of this step seeks to \textit{cancel} some projections if they are on shared boundaries.  That is, 
we aim to find a set of sub-feature topological projections $\{\mathcal{P}^{l}_j\}_{(l,j)\in\mathcal{A}}$, where $\mathcal{A}$ is the index set of the potential-features that may contribute to a true feature,  such that the module-$2$ sum of the $\mathcal{P}^{l}_j$s equals zero and forms a $r^l_j$-dimensional feature. In practice, if the module-2 sum is less than some (small) projection error, the projections can be merged.  See section \ref{sec-app:merge method} for more details.  

We remark that this method is particularly useful in theoretically achieving linearized computational complexity as $n\to\infty$ and when $m$ is small. On the other hand, when $m$ is large, the Projected Merge method can be unstable since it only detects a true feature when all of the sub-features that make up the true feature exist. Therefore, we propose the following Representative Rips Merge method as a stable merge method, which is what we primarily use in the empirical studies of Section~\ref{sec:simulation} and in the application in Section~\ref{sec:app}.

\paragraph{Representative Rips Merge method}

For the Representative Rips Merge method, the distance between potential-features is computed as the Euclidean distance between their nearest representative data points:
\begin{align}
    \rho(\{y_{j}^l(i)\}_{i\in[n_{j}^l]},\{y_{j'}^{l'}(i')\}_{i'\in[n_{j'}^{l'}]})=\min_{i,i'} \rho(y_{j}^l(i),y_{j'}^{l'}(i')),
\end{align}
for all distinct potential-features $\{y_{j}^l(i)\}_{i\in[n_{j}^l]}$ and $\{y_{j'}^{l'}(i')\}_{i'\in[n_{j'}^{l'}]}$. Then, the constructed distance matrix is used to compute a VR filtration and persistence diagram. The representative data points recovered from the resulting persistence diagram of the constructed distance matrix can identify which potential-features can be merged as a true feature based on their appearance on the persistence diagram.


\subsubsection{Merged Birth and Death Time Estimation}\label{sec:method:estimate recovery}

After the merge step, one possibility is to run a VR filtration on the merged features to estimate the merged feature's birth and death time.\footnote{Note that the DaC-estimated persistence diagram includes the merged persistence diagram and the complete diagram (i.e., the features that do not form with the supplemental boundary).} Since the data size of each feature is much smaller than $n$, recomputing a VR filtration is generally not expensive. We call this resulting persistence diagram the \textit{Naive estimate}, defined as 
\begin{align}\label{eq-def:naive estimate}
    \hat{\D}\defeq \{(\hat{r}_{j'}, \hat{b}_{j'}, \hat{d}_{j'}): j' = 1, \ldots, |\hat{\D}|\} \cup \Delta,
\end{align}
where the $(\hat{r}_{j'}, \hat{b}_{j'}, \hat{d}_{j'})$ are from the recomputed persistence diagram on the subset of representative points.
An alternative approach that avoids recomputing a VR filtration is 
using potential-diagrams to estimate the persistence diagram of the merged data as 
\begin{equation}
    \D^C\defeq \{(r^{C}_{j'}, b^{C}_{j'}, d^{C}_{j'}): {j'} = 1, \ldots, |\D^C|\} \cup \Delta,
\end{equation}
where the $(r^{C}_{j'}, b^{C}_{j'}, d^{C}_{j'})$ are estimated using the combined potential-features' representative data, $\{y^l_{j}(i)\}_{i\in[n_{j}^l]}$, and their persistences $(r_j^l,b_j^l,d_j^l)$, where $(l,j)\in\mathcal{A}_{j'}$. These combined data are denoted as $\{y_{i}^C\}_{i}^{j'} \defeq \cup_{(l,j)\in\mathcal{A}_{j'}} \{y^l_{j}(i)\}_{i\in[n_{j}^l]}$. In the following, for simplicity, we write $\{y_{j}^C\}_{j}^{j'}$ as $\{y_{i}^C\}_{i}$. Then a birth time estimate, $b^{C}_{j'}$, can be computed as
\begin{equation}\label{eq:birthEsti}
  b^{C}_{j'} = \max_{(l,j)\in\mathcal{A}_{j'}}\{b^l_j\}.
\end{equation}
This is a reasonable estimate for the birth time because the birth time of the true feature is close to the maximum birth time of the potential-features that are merged, as the birth time is a local property.
However, this approach does not work for the death time estimate because the death time relies on global information of the feature. Therefore, we consider estimating the death time by only using a specific subset of data of the feature $\{y^C_i\}_i$ and approximating the $k$-dimensional feature as a $k$-sphere. To estimate the death time, we consider all the $(k+1)$-simplexes using $(k+2)$ representative points $y^C_{i_1},y^C_{i_2},y^C_{i_3},\dots,y^C_{i_{k+2}}\in\{y_i^C\}_i$ as
\begin{align}
    \hat{y}^C_{i_1},\hat{y}^C_{i_2},\hat{y}^C_{i_3},\dots,\hat{y}^C_{i_{k+2}}=\underset{y^C_{i_1},y^C_{i_2},y^C_{i_3},\dots,y^C_{i_{k+2}}\in\{y_i^C\}_i}{\mathrm{argmax}}\left(\min_{y^C_{i_s},y^C_{i_t}\in \{y^C_{i_1},\dots,y^C_{i_{k+2}}\}}\rho(y^C_{i_s},y^C_{i_t})\right).
\end{align}
Then, we consider the following death time estimate:
\begin{align}\label{eq:deathEsti}
    d^C_{j'}= \max_{\hat{y}^C_{i_s},\hat{y}^C_{i_t}\in \{\hat{y}^C_{i_1},\dots,\hat{y}^C_{i_{k+2}}\}}\rho(\hat{y}^C_{i_s},\hat{y}^C_{i_t}).
\end{align}
Equation~\ref {eq:deathEsti} seeks the largest equilateral $k$-simplex that fits in the feature and use the maximum length among its edges to represent the death time, which is a good approximation if the feature is a $k$-sphere; this formula requires a memory of $\OO(1)$ and a computational complexity of $\OO(n^{k+2})$. This is a consistent estimator for a $k$-dimensional sphere, but may be larger than the true death time for other shapes. We call the estimates based on Equations \eqref{eq:birthEsti} and \eqref{eq:deathEsti} as \textit{pointwise estimates}.

\subsection{Computational Complexity and Memory} \label{sec:memory}
Pseudo-code of the proposed DaC is displayed in Algorithm \ref{algorithm:DaC}. Assume the maximum dimension of a feature to recover is $k$. A natural approach for partitioning the data space is to divide the whole space into $m$ equally-sized hyperrectangles, as in Figure \ref{fig:SimpleExample}. Other designs of sub-regions can also be considered if density information or certain details on the embedding space are known.  For example, a reasonable alternative is to divide the data so that the number of data points is similar in each sub-region. The choice of $m$ has a trade-off between accuracy and computational cost such that a smaller $m$ has higher accuracy but also a higher computational cost. In general, selecting $m$ such that $2\le m \le \OO\left(n/\log(n)\right)$ ensures a high probability of recovering the feature; more details are provided in Section~\ref{sec:theory}.
Assuming the data are evenly distributed in each sub-region, then the computational complexity to compute a persistence diagram up to and including homology dimension $k$ decreases from $\OO(n^{3k+6})$ with the traditional persistence diagram computation to $\OO\left(m\left(n/m\right)^{3k+6}\right)$ with the proposed DaC. 
In the worst-case scenario with a memory of $\OO(n^{2k+3})$, DaC can reduce the memory to $\OO\left(\left(n/m\right)^{2k+3}\right)$. 
If $m=\OO\left(n/\log(n)\right)$, then DaC's computational complexity is (almost) linear as $\OO\left(n\left(\log(n)\right)^{3k+5}\right)$ with a memory $\OO((\log n)^{2k+3})$.  

The merging step is not trivial algorithmically or computationally. The Projected Merge method takes $\OO(m)$ steps to merge, and the Representative Rips Merge method has a worst-case computational complexity of $\OO(m^{3k+6})$ to merge the potential-features. With $m\sim n/\log(n)$, the computational complexity for DaC with the Projected Merge method is $\OO(m+n^{3k+6}/m^{3k+5})$ scaling (almost) linearly to $n$. The computational complexity for the Representative Rips Merge method is $\OO(m^{3k+6}+m\left(n/m\right)^{3k+6})$, and is optimal when $m\sim n^{\frac{3k+6}{6k+11}}$. Fortunately, we do not see the worst computational complexity showing up in the simulation. On the other hand, one could apply the DaC iteratively to improve the computational complexity of the Representative Rips Merge method.
	\begin{algorithm}[H]
	\begin{algorithmic}
	\State \textbf{Input: } $Y_{1:n}$ (data), $k$ (homology dimension), $m$ (number of sub-regions)
 \State \textbf{Input (optional):} Partition design $\P = \cup_{l=1}^m\P^{(l)}$ (default: equal-sized sub-regions)
	\State \textbf{Output:} Estimated persistence diagram of homology dimension $k$ 
 \If{$\P=\emptyset$}{ Create $m$ equally-sized sub-regions with supplemental boundary, $\P$} 
 \EndIf{}
 \State Define:  Out$\leftarrow \emptyset$ (store complete features), PF$\leftarrow \emptyset$ (store potential-features)
    \For{l in 1:m}
        \State{Step 1:} For $Y_{1:n}\cup (\B^{(l)} \in \P^{(l)})$, compute the distance matrix, Dist$^{(l)}$ (\S\ref{sec:method:distance matrix construction}) 
        \State{Step 2:} Compute persistence diagram of Dist$^{(l)}$, $\D^{l}$
        \State{Step 3:} Out $\leftarrow$ Out$\cup \D^l_{\text{complete}}$
        \State{Step 4:} PF $\leftarrow$ PF$\cup \D^l_{\text{potential}}$
    \EndFor{}
 \If{PF$\neq\emptyset$}{ Merge potential-features (\S\ref{sec:method:combination method})}, estimate birth and death times (\S\ref{sec:method:estimate recovery}), $\hat{\D}$ or $\D^C$, and Out $\leftarrow$ Out$\cup (\hat{\D} \text{ or } \D^C$)
 \EndIf{}
\State \Return{Out}
\end{algorithmic}
\caption{DaC Persistent Homology Recipe}
\label{algorithm:DaC}
\end{algorithm}

\subsection{Theoretical Results}\label{sec:theory}

This section highlights the main theoretical results for the proposed DaC method to estimate a persistent diagram.  Let $ \M $ be a $\mo$-dimensional \textit{closed  and convex feature} embedded in $ \R^{D} $. We defer the concrete mathematical assumptions and proofs to the Appendix. Convexity of the feature is assumed to avoid the ambiguity of the number of features.

Let $Y_{1:n} = \{y_i\}_{i=1}^n \in \R^{D}$ be data sampled independently from a distribution $\mu$ supported on $\M$. The distribution $\mu$ is of the form
\begin{equation}\label{eq:dmu}
    \dd\mu = \phi\ \dd \vol_\M,
\end{equation}
for density function $\phi:\M\to \R_+$. 
The following Assumption \ref{assumption:probability} guarantees that there will not be a too sparse region in the data distribution.
\begin{restatable}{thm}{Assum}\label{assum:phi}
\begin{assumption} \label{assumption:probability} For density function $\phi$ from Equation~\eqref{eq:dmu}, 
    \begin{equation}\label{eq:assum-phi_n}
    \phi(y)\geq  c_\phi,  y\in\M,
\end{equation}
for some positive constant $c_\phi$.
\end{assumption}
\end{restatable}

The following theorem establishes a trade-off between the error of the birth and death time estimates and the recovery probability of the feature. 
\begin{theorem}\label{thm:informal}
    With respect to the true diagram $\D^0=\{(k, b^0, d^0)\}$ computed using the full dataset $\Yn$, and under some mild assumption on the data partitioning, there exist representative data points of potential-features $\{y_{j}^l(i)\}_{i\in[n_{j}^l]}$ where $(l,j)\in \mathcal{A}$ which can be merged as a $k$-dimensional feature with naive birth and death time estimators (Equation~\eqref{eq-def:naive estimate}) denoted as $(k,\hat{b},\hat{d})$. For a given error, $\delta$, such that $0< \delta \le c m^{-\frac{1}{k}}$, with probability at least $1-\frac{C}{\delta^\mo}\exp(-{cn\delta^{\mo}})$, 
    \begin{equation}
      |b^0-\hat{b}|\le \delta ;\quad  |d^0-\hat{d}|\le \OO(\delta),
    \end{equation}
    where constants $C>1$, $0<c<1$ only depends on intrinsic properties of $\M$ and $\phi$.
\end{theorem}

\begin{sloppypar}
When the error bound $\delta$ scales as $m^{-\frac{1}{\mo}}$, the recovery probability in Theorem \ref{thm:informal} is $1-Cm\exp(-{cn/m})$. Consider the case of $m=\Omega(1)$, then the recovery probability is exponential in $n$. If $m$ is chosen as $\Omega(n/\log(n))$, the computational complexity is linear with $n$ and the recovery probability is $1-1/(c\log(n))$. This shows that DaC can scale as a linear algorithm as \cite{Sheehy2012Linear}; however, our result shows the trade-off between $m$ and the recovery probability explicitly. Therefore, one may choose an $m$ to balance computation and approximation considerations.  Furthermore, when $n\to\infty$, the error of birth and death time estimates converges to $0$, while the approach in \cite{Sheehy2012Linear} does not guarantee this.
\end{sloppypar}

The previous result is for a single manifold, but we also provide a multi-manifold generalization detailed in section \ref{sec:multiple}. Assume $\M=\cup_{l=1}^K\M_l$ where $\M_l$ is $\mo$-dimensional closed and convex manifold, and that the $\M_l$ are separable in the sense that
\begin{align}
    \min_{l\not= l'}\dist(\M_l,\M_{l'})\ge \kappa,
\end{align}
where $\kappa>0$ is the smallest distance between the $\M_l$. 
This assumption is provided to prevent the ambiguity of the number of features.\footnote{For more details about the properties of the features, see section \ref{sec:multiple}.}
The distribution $\mu$ is assumed to be a mixture model taking the form
	\begin{equation}
	    d \mu= w_1\phi_1 d\vol_{\M_1}   + \dots +w_K \phi_K d \vol_{\M_K},
	\end{equation}
where $\dd\vol_{\M_l}$ denotes integration with respect to the Riemannian volume form associated with $\M_l$.  Then, we obtain the following result.

\begin{theorem}
    Under mild assumption on the DaC construction and the value of $m$, when $0< \delta \le c m^{-\frac{1}{k}}$, with probability at least $1-\frac{C}{\delta^{\mo}}\exp(-{cn\delta^{\mo}})$,  
    \begin{align}
      |b_l-\hat{b}_l|\le \delta ;\quad  |d_l-\hat{d}_l|\le \OO(\delta),
    \end{align}
    where $b_l$ and $d_l$ are the birth and death times of $\M_l$, respectively, $\hat{b}_l$ and $\hat{d}_l$ are the corresponding naive birth and death time estimates, and $C>1$, $0<c<1$ only depend on the intrinsic properties of $\M$.
\end{theorem}

\section{Simulation Study}\label{sec:simulation}

\subsection{Two Circles Example}\label{sec:simulation:two circle cases}

To illustrate the performance of the proposed DaC method, we consider a two-feature example. The data are generated with 413 data points such that the density of points on different features is the same, but the features are of different sizes; see an example in Figure \ref{fig:1}. Persistence diagrams are estimated using the k-means++ \citep{moitra2018cluster}\footnote{According to the implementation in \url{https://github.com/wilseypa/lhf}.} and DaC algorithms with varying parameter values of the number of clusters $k_c$ and the number of sub-regions $m$, respectively. 
For the proposed DaC algorithm, unless otherwise specified, the Representative Rips Merge method and pointwise estimates (Section~\ref{sec:method:estimate recovery}) are used. The data regions are split into $m$ equal-sized sub-regions, and the DaC algorithm is applied as described in Algorithm \ref{algorithm:DaC}.  First the two methods are compared for a single dataset, then we consider 100 iid realizations of the same data-generating process where the methods are compared based on the frequency of correctly detecting the features, and on the bottleneck distance between the estimated and true persistence diagrams.

Recall that the cluster-based algorithm \citep{moitra2018cluster} uses the k-means++ algorithm to select $k_c$ points on which to compute the persistence diagram. To detect two features in the persistence diagram for the data displayed in  Figure \ref{fig:1}, the k-means++ requires $k_c\geq 90$, which is the same order as the total number of data points. This is not surprising because the unbalanced sample size influences the efficiency of k-means++ making it challenging to detect the smaller feature. This illustrates that the performance of the cluster-based algorithm depends on the preliminary clustering algorithm and suffers from its inability to see the small-scale topological features. 
On the other hand, the proposed DaC algorithm only needs $m=15$, with each sub-region containing approximately 25 points to detect the two features. Therefore, the running time and required memory of DaC are better than k-means++ when conditioning on both methods detecting the correct number of features.

For additional comparisons between the cluster-based algorithm of \cite{moitra2018cluster} and the DaC method, we generate 100 iid datasets as in Figure~\ref{fig:1}\footnote{The implementation can be found in \href{https://anonymous.4open.science/r/Divide-and-Conquer-method-in-TDA-603F/}{https://anonymous.4open.science/r/Divide-and-Conquer-method-in-TDA-603F/}}, and test different sets of parameters of the k-means++ algorithms and the DaC algorithm with varying parameter values for $k_c$ and $m$. The proposed DaC successfully recovers the birth and death time estimates even when $m=225$, such that the average bottleneck distance to the true birth and death times is small ($0.0078$). When $k_c=103$, the averaged bottleneck distance to the true birth and death times is $0.0103$. To compare the algorithms under the same maximum memory with $k_c=25$ in k-means++ and $m=64$ in DaC, then the average bottleneck distance of k-means++ to the true birth and death times is $0.1015$, which is worse than the accuracy of DaC estimates with an average bottleneck distance of 0.0061.
%
%
\begin{figure}[htbp]
\centering
			\centering    	        \includegraphics[width=5.5cm]{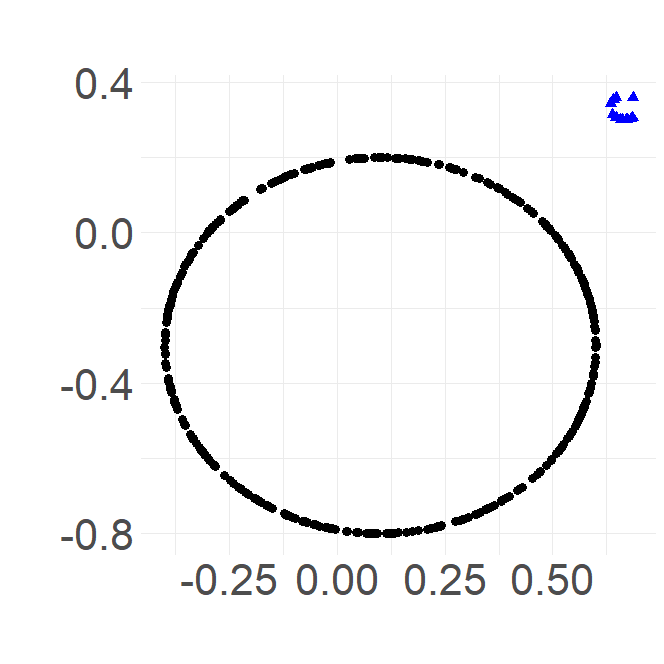}
			\caption{Two circles with radius $1$ (black points) and $1/30$ (blue triangles); the number of data points on each circle is 400 and 13, respectively. The goal is to recover two $H_1$ features on a persistence diagram and estimate their birth and death times.}
			\label{fig:1}
	\end{figure}

Figure~\ref{fig:bottleneck distance with kmeans++} displays the bottleneck distance results under the varying parameter values. The choice of $k_c$ represents 6\%, 10\%, 25\%, 50\% and 75\% of the sample size $n$, similar to the suggested values in \cite{moitra2018cluster}.  The bottleneck distance of the proposed DaC algorithm is significantly better than the k-means++ algorithm when the number of clusters $k_c$ is small. It is not until $k_c$ equals almost the same order as the total number of data points ($n=413$) that the bottleneck error of k-means++ becomes competitive with DaC, but involves a much higher memory cost.
\begin{figure}[htbp]
\centering
	\includegraphics[width=9cm]{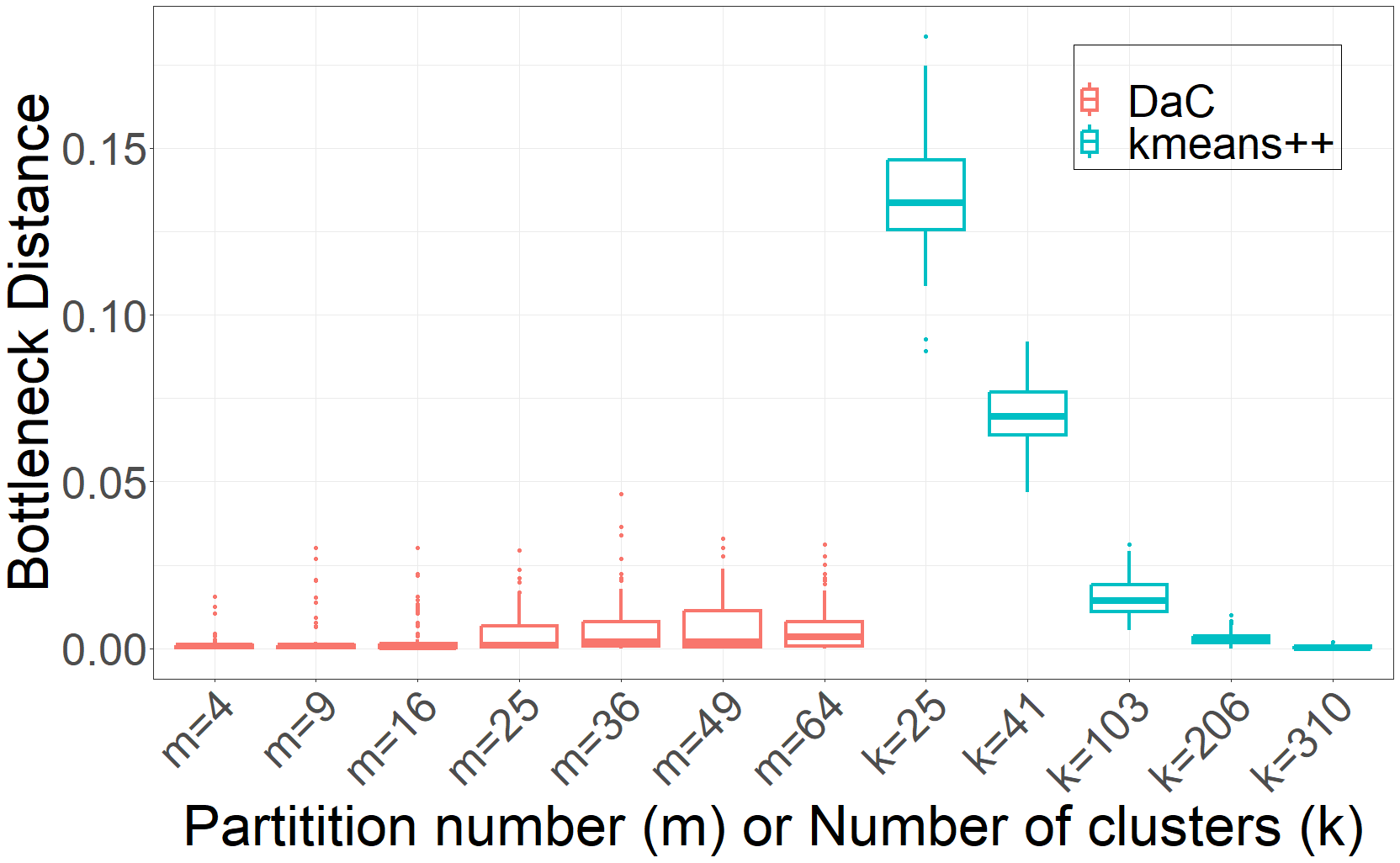}
			\caption{The bottleneck distance between the true persistence diagram and the proposed DaC method or k-means++ methods based on 100 i.i.d. realizations of the two-circle data from Figure~\ref{fig:1}, with different values for $m$ and $k_c$. 
   }
			\label{fig:bottleneck distance with kmeans++}
	\end{figure}

To compare the memory, we count the number of points in each sub-region and the total number of sub-regions that contain at least one point, and use this to approximate the maximum memory size and the computational complexity for the proposed DaC. As a comparison, the number of clusters $k_c$ determines the memory size for k-means++, which is $\OO(k_c^5)$ for $k=1$ corresponding to an $H_1$ feature, see Section~\ref{sec:memory}. Recall that the memory and computational complexity are polynomial in the number of points used in VR filtration. As seen from Figures \ref{fig:bottleneck distance with kmeans++}, \ref{fig:MaxPointsWithinSubregions} and \ref{fig:Total Representative sub-regions}, the memory and computational complexity of the proposed DaC with $m=4$ approximately equal the k-means++ memory and computational complexity with $k_c=103$, but the bottleneck distance for the birth and death time estimate errors is much lower for the DaC method. When $m$ becomes larger in the proposed DaC method, the bottleneck distance to the true estimates only mildly increases, while the reduction in memory and computational complexity is substantial.
\begin{figure}[htbp]
\par
		\centering
		\begin{subfigure}[t]{0.45\textwidth}
			\centering
    	         \includegraphics[width=6cm]{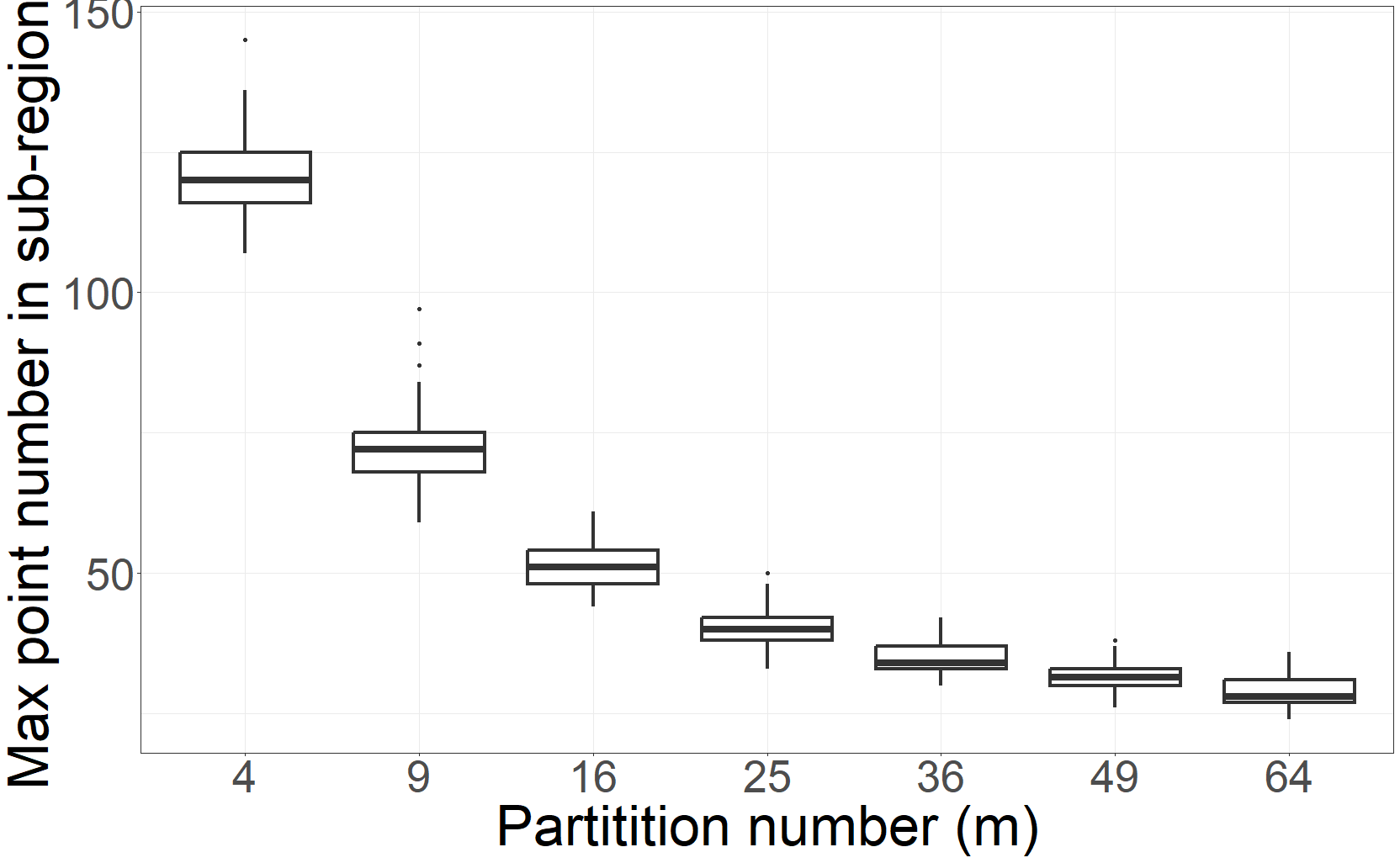}
			\caption{Max sample size within each sub-region}
			\label{fig:MaxPointsWithinSubregions}
			\end{subfigure}
		\begin{subfigure}[t]{0.45\textwidth}
 \centering  	        \includegraphics[width=6cm]{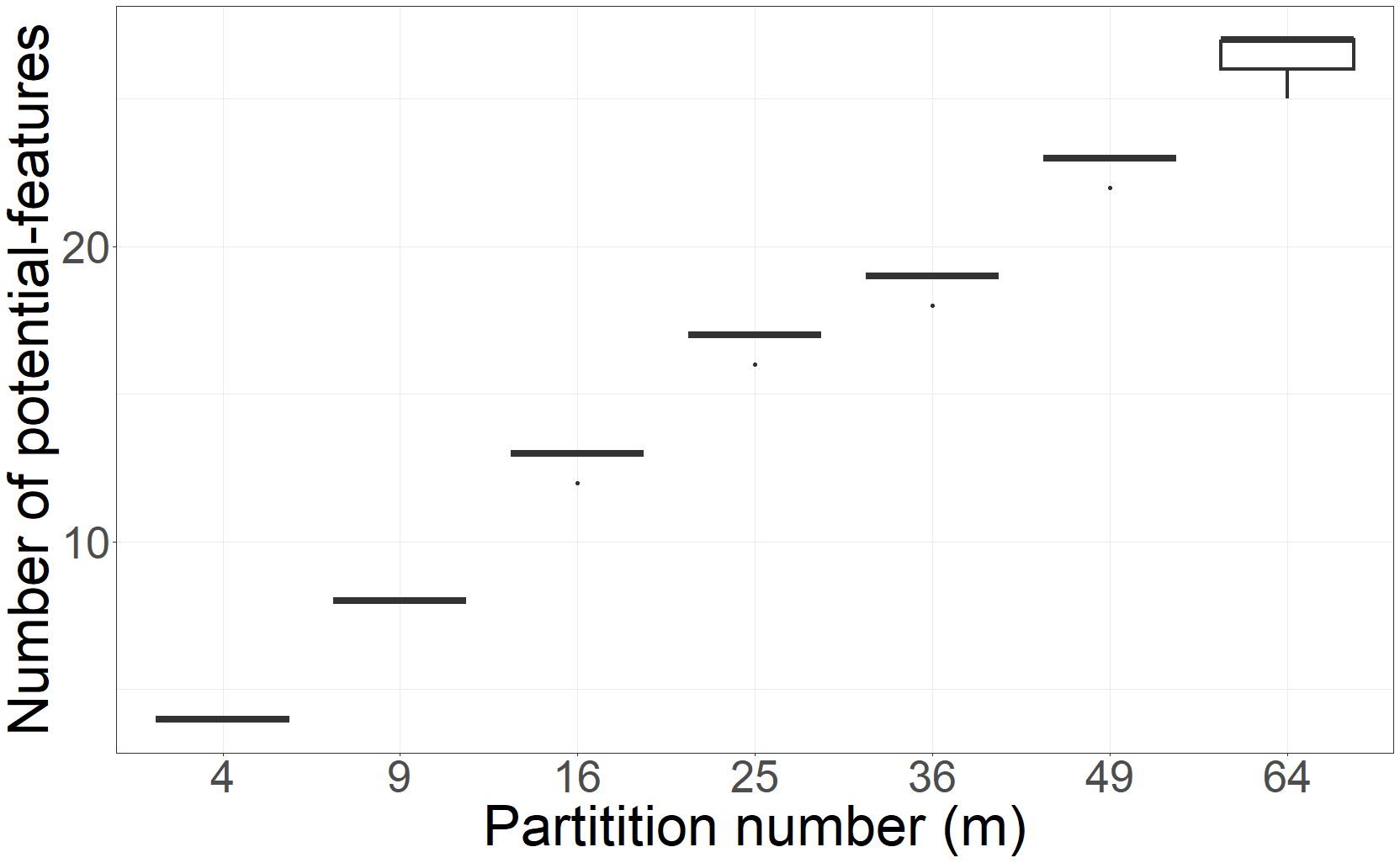}
			\caption{Total number of potential-features}
			\label{fig:Total Representative sub-regions}
   \end{subfigure}
   \caption{The memory required for the proposed DaC method for a given $m$ is the 5th power of the larger value between the maximum number of data points within a sub-region (a)  and  the total number of potential-features (b); see Section~\ref{sec:memory}.}
	\end{figure}

\subsection{Computation and Memory}

\subsubsection{Representative Rips Merge method}\label{sec:Representative Rips Merge method}
We test the recovery probability rate of the $H_1$ feature and the bottleneck distance with respect to different sample sizes and the choices of the total number of sub-regions, $m$. The two proposed merge methods, the Representative Rips Merge method and the Projected Merge method, are evaluated in this section and Section~\ref{sec:eval_pmm}, respectively. The underlying feature is a one-dimensional circle, and we generate $n$ data points uniformly from the circle, and then run DaC on the data points. The circle is divided into $m$ evenly-sized sub-regions.  The $n$ is set to 100, 200, 400, 800, and 1600, and $m$ is set to 4, 16, 64, 256, and 1024 for the Representation Rips merge method. The $m$ is the total number of sub-regions, but not the number of potential-features; the total number of potential features for each $m$ are roughly 4, 8, 16, 32, and 64, respectively; these numbers can be used to compute the (reduced) average memory of the proposed DaC. For example, if $m=16$ the memory is reduced from $\Omega(n^5)$ to $\Omega\left(\max\{n^5/8^5,8^5\}\right)$.
The average bottleneck distance of the Representative Merge method is displayed in Table \ref{table:bottleneck distance} with the corresponding recovery rates in Table \ref{table:recovery}.

From Theorem \ref{thm:informal}, the scaling between the error $\delta$ and the sample size $n$ is $\delta\sim 1/n$ when $k=1$ to guarantee a high recovery probability; $\delta$ can be thought of as the upper-bound on the bottleneck distance.  From Table \ref{table:bottleneck distance}, the empirical scaling between $\delta$ and $n$ is close to, or slightly faster than, the theoretical scaling (linear rate).
\begin{table}[h!]
\centering
\begin{tabular}{l|r|r|r|r|r} 
\hline
 &  m=4 &  m=16 & m=64 &m=256 & m=1024  \\
\hline
n=100 & 0.014 (0.021) & 0.112 (0.064) & 0.235 (0.129) & $\diagup$ & $\diagup$ \\
\hline
n=200 & 0.007 (0.013) & 0.023 (0.030) & 0.095 (0.043) & 0.117 (0.045) & 0.157 (0.077) \\
\hline
n=400 & 0.002 (0.005) & 0.004 (0.007) & 0.028 (0.021) & 0.043 (0.022) & 0.073 (0.018) \\
\hline
n=800 & 0.001 (0.001) & 0.002 (0.004)  & 0.007 (0.009) & 0.013 (0.012) & 0.026 (0.011) \\
\hline
n=1600 & 0.000 (0.001) & 0.000 (0.001) &  0.001 (0.003) & 0.004 (0.005) &  0.008 (0.005)\\
\hline
\end{tabular}
\caption{Uniformly generated data points from a one-dimensional circle with radius $1$. The average bottleneck distance (based on 100 iid repetitions) between the DaC persistence diagram using the Representative Rips Merge method and the true persistence diagram, and their standard deviation within the brackets. When $n=100$, DaC persistence diagrams are not computed for $m=256$ or $1024$ because the sample size is too small to distribute into the sub-regions.}
\label{table:bottleneck distance}
\end{table}

\begin{table}
\centering
\begin{tabular}{l|r|r|r|r|r} 
\hline
 &  m=4 &  m=16 & m=64 &m=256 & m=1024 \\
\hline
n=100 & 100\% & 98\% & 63\%  & $\diagup$ &$\diagup$ \\
\hline
n=200 & 100\% & 100\% & 100\% & 100\% & 85\% \\
\hline
n=400 & 100\% & 100\% & 100\% & 100\% & 100\%\\
\hline
n=800 & 100\% & 100\% & 100\% & 100\% & 100\% \\
\hline
n=1600 & 100\% & 100\% & 100\% & 100\% & 100\% \\
\hline
\end{tabular}
\caption{The recovery rates of  uniformly generated data points from a one-dimensional circle with radius 1 using the Representative Rips merge method, based on 100 iid realizations of the data.}
\label{table:recovery}
\end{table}

\subsubsection{Projected Merge method} \label{sec:eval_pmm}
The Projected Merge method is evaluated with $m=2$ for equal-sized sub-regions. We repeat the same data generation setup as in Section~\ref{sec:Representative Rips Merge method}. In this scenario, the topological projection introduced in Section~\ref{sec:method:combination method} reduces to an orthogonal projection. When $n=100,200,400,800,$ and $1600$, the recovery rates are $100\%$ with bottleneck distances of 0.007, 0.003, 0.001, 0.000, 0.000, respectively. 

\subsubsection{Two-Dimensional Sphere}
As a demonstration of DaC in the three-dimensional setting, a two-dimensional sphere dataset is considered using the Representative Rips Merge method. The sample size is fixed at $n=1200$, and 100 iid repetitions of the data are drawn uniformly from a radius-$1$ two-dimensional sphere. The total region is partitioned into $m=3^3$, $4^3$, $5^3$, and $6^3$ equal-sized sub-regions. The recovery probability and the accuracy are displayed in Table \ref{table:sphere-accuracy}. The running memory limit is restricted to $10$ GB. 

As shown in Table \ref{table:sphere-accuracy}, the recovery rate ranges from 75\% to 90\% due to the small sample size.\footnote{We attempted to use a sample size of $n=1600$, but could not compute the true persistence diagram (using the full dataset) up to homology dimension two, even on a high-performance computing cluster with $100$ GB memory limit or with Ripser.  This highlights the computational challenges of computing persistence diagrams using a Rips filtration.} As $m$ increases, the average bottleneck distance becomes larger as expected. In this example, the R TDA package requires too much memory to produce the true diagram, so we utilize Ripser \citep{bauer2021ripser} instead.

\begin{table}[h!]
\centering
\begin{tabular}{|l|r|r|r|r|} 
\hline
 &  $m=3^3$ &  $m=4^3$ & $m=5^3$ &$m=6^3$   \\
\hline
Recovery Probability& 90\% & 85\% & 75\% &79\%\\
\hline
Average bottleneck distance& 0.117 (0.177)& 0.203 (0.233) & 0.225 (0.251) & 0.389 (0.238) \\
\hline
\end{tabular}
\caption{We sample $1200$ points uniformly from a radius-$1$ two-dimensional sphere, with the goal of identifying the $H_2$ feature. The empirical recovery rate and average bottleneck distance (and standard deviation in parentheses) between the DaC persistence diagram and the true persistence diagram are shown, based on 100 iid repetitions of the data.}
\label{table:sphere-accuracy}
\end{table}

\section{Wisconsin Lake Data}\label{sec:app}
Lakes are an important natural resource studied by limnologists.
\cite{garn2003study} highlights studies on lakes in Wisconsin, and notes that the spatial distribution of lakes in the state is particularly unique:  northern Wisconsin contains some of ``the densest clusters of lakes found anywhere in the world'', which they attribute to glaciers from 10,000 years ago.
In this section, we investigate the spatial distribution of lakes in Wisconsin using persistent homology, which can be used to summarize the spatial distribution for visualization or comparison between other states or regions.  Here, we focus on computing a persistence diagram for a dataset containing 16,683 lakes in Wisconsin; see Figure \ref{fig:lake_data}.

The proposed DaC method is used to estimate the persistence diagram of the Wisconsin lake data, which are freely available through the Wisconsin Department of Natural Resources\footnote{The data is derived from \url{https://apps.dnr.wi.gov/lakes/lakepages/Results.aspx}.}. The full dataset includes 16,711 lakes, but some preprocessing was required, as follows.  Thirteen lakes with positive longitudes were corrected to be negative.  Forty-three lakes had latitude and longitudes of zero, of which 15 could be identified using their Waterbody ID Code on the Wisconsin Department of Natural Resources website's ``Find a Lake'' site\footnote{\url{https://apps.dnr.wi.gov/lakes/lakepages/Search.aspx}} and observed on a map using their updated latitude and longitude coordinates; the other 28 lakes (0.17\% of total) are excluded from our analysis resulting in 16,683.  We did not set a minimum lake size, though 3,319 of the 16,683 lakes (19.9\%) have a size below one acre and may more appropriately be considered ponds.
In summary, the latitude and longitude coordinates of 16,683 lakes in Wisconsin were used in this analysis.

The aim is to recover the $H_1$ features from the spatial distribution of lakes. Unfortunately, due to the large sample size, the classical persistent homology algorithm cannot directly run on these data with a 100 GB memory limitation. Instead, the proposed DaC method is employed with $16\times 16$ uniform-area sub-regions (such that some sub-regions extended beyond the boundary of the state) using the Representative Rips Merge method. Since each sub-region is small, the smallest birth time among the merged sub-features and the largest death time among merged sub-features are used as the birth and death time estimates, respectively. We remark that this can be slightly conservative compared to the true estimate, but the error is small since $m$ is large. The locations of the Wisconsin lakes are displayed in Figure \ref{fig:lake_data} and the DaC persistence diagram is shown in Figure \ref{fig:lake_PD}. The DaC persistence diagram detected 2,346 $H_1$ features, including 234 merged $H_1$ features.
\begin{figure}[htbp]
		\par\medskip
		\centering
		\begin{subfigure}[t]{0.4\textwidth}
			\centering
    	        \includegraphics[width=5cm]{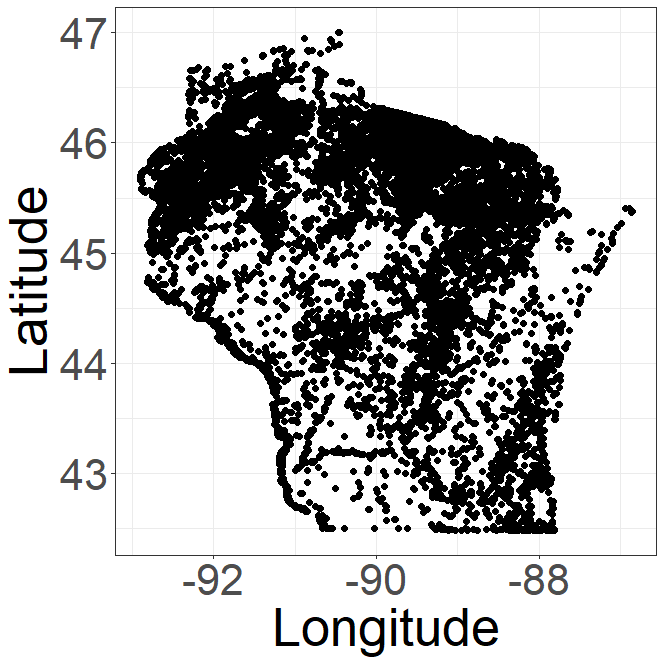}
			\caption{WI Lake Locations}
			\label{fig:lake_data}
			\end{subfigure}
		\begin{subfigure}[t]{0.4\textwidth}
			\centering
			\includegraphics[width=5cm]{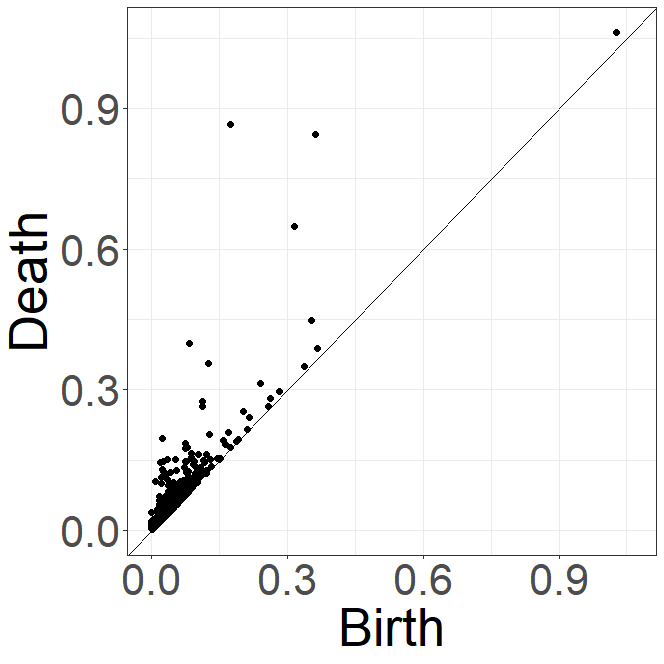}
			\caption{DaC $H_1$ Persistence Diagram}
			\label{fig:lake_PD}
		\end{subfigure}
		\caption{Wisconsin Lakes.  (a) Each black point represents one of the 16,683 Wisconsin lakes used to compute the DaC persistence diagram displayed in (b).  (b) The DaC (merged) $H_1$-persistence diagram of the Wisconsin Lakes data.  There are 2,346 $H_1$ features detected, including 234 merged features. }
	\end{figure}

\section{Discussion and Concluding Remarks} \label{sec:conclusion}

Persistent homology is useful for inference, prediction, and visualization in many applications.  However, the computation of a persistence diagram, especially using a Rips filtration, is memory-intensive which limits its applicability to even moderately-sized datasets.  To help mitigate these computational limitations, a DaC method is proposed and justified empirically and theoretically for the Rips filtration setting. Under some assumptions on the underlying manifold, we provide an explicit trade-off between the computational complexity and memory and the accuracy of the estimated birth and death times. When the estimation error $\delta=cm^{-\frac{1}{k}}$, with $m\sim n/\log(n)$, the computational complexity scales linearly as $n$ increases while the recovery probability of the underlying feature scales $1-1/n$; with $m$ scaling as a constant, the computational complexity of DaC scales in constant time better than the standard persistent homology while the recovery probability scales $1-\exp(-n)$. This theoretical trade-off makes the proposed DaC practical for applications. Furthermore, we verify that the empirical rate matches the theoretical rate for the bottleneck distance between the estimated and true birth and death times. The proposed DaC method was empirically found to scale better in terms of memory and computational complexity than the $k$-means++ cluster-based algorithm in the settings considered. Finally, we estimated a persistence diagram using DaC for a Wisconsin Lakes dataset in which the traditional persistence diagram could not be computed due to memory limitations.

A benefit of the proposed DaC is that small, medium, and large-scale features can be recovered.  We present a general divide-and-conquer framework with two options for merging potential-features, though these merging approaches can be improved. First, the Projected Merge method  is currently only suitable for small $m$, so an algorithm for this topological projection method for large $m$ is still needed.  Second, the point-wise birth and death time estimates are not theoretically justified given the merging representative data points, but the empirical performance is reasonable. 
A general method to estimate the birth and death times of a feature directly from a representative sample without recomputing the filtration for a general feature (i.e., not assuming a $k$-sphere as in Section~\ref{sec:method:estimate recovery}) remains an open problem.
Addressing these open challenges is important in understanding the merging process among the sub-features.

\bibliographystyle{unsrtnat}
\bibliography{references}  

\begin{thebibliography}{67}
\providecommand{\natexlab}[1]{#1}
\providecommand{\url}[1]{\texttt{#1}}
\expandafter\ifx\csname urlstyle\endcsname\relax
  \providecommand{\doi}[1]{doi: #1}\else
  \providecommand{\doi}{doi: \begingroup \urlstyle{rm}\Url}\fi

\bibitem[{Center for High Throughput Computing}(2006)]{chtc}
{Center for High Throughput Computing}.
\newblock Center for high throughput computing, 2006.
\newblock URL \url{https://chtc.cs.wisc.edu/}.

\bibitem[Bubenik et~al.(2015)]{bubenik2015statistical}
Peter Bubenik et~al.
\newblock Statistical topological data analysis using persistence landscapes.
\newblock \emph{Journal of Machine Learning Research}, 16\penalty0
  (1):\penalty0 77--102, 2015.

\bibitem[Reininghaus et~al.(2015)Reininghaus, Huber, Bauer, and
  Kwitt]{reininghaus2015stable}
Jan Reininghaus, Stefan Huber, Ulrich Bauer, and Roland Kwitt.
\newblock A stable multi-scale kernel for topological machine learning.
\newblock In \emph{Proceedings of the IEEE conference on computer vision and
  pattern recognition}, pages 4741--4748, 2015.

\bibitem[Anirudh et~al.(2016)Anirudh, Venkataraman, Ramamurthy, and
  Turaga]{Anirudh2016}
Rushil Anirudh, Vinay Venkataraman, Karthikeyan~Natesan Ramamurthy, and Pavan
  Turaga.
\newblock A riemannian framework for statistical analysis of topological
  persistence diagrams.
\newblock In \emph{29th IEEE Conference on Computer Vision and Pattern
  Recognition Workshops, CVPRW 2016}, pages 1023--1031. IEEE Computer Society,
  2016.

\bibitem[Adams et~al.(2017)Adams, Emerson, Kirby, Neville, Peterson, Shipman,
  Chepushtanova, Hanson, Motta, and Ziegelmeier]{adams2017persistence}
Henry Adams, Tegan Emerson, Michael Kirby, Rachel Neville, Chris Peterson,
  Patrick Shipman, Sofya Chepushtanova, Eric Hanson, Francis Motta, and Lori
  Ziegelmeier.
\newblock Persistence images: A stable vector representation of persistent
  homology.
\newblock \emph{Journal of Machine Learning Research}, 18, 2017.

\bibitem[Robinson and Turner(2017)]{robinson2017hypothesis}
Andrew Robinson and Katharine Turner.
\newblock Hypothesis testing for topological data analysis.
\newblock \emph{Journal of Applied and Computational Topology}, 1\penalty0
  (2):\penalty0 241--261, 2017.

\bibitem[Kusano et~al.(2017)Kusano, Fukumizu, and Hiraoka]{kusano2017kernel}
Genki Kusano, Kenji Fukumizu, and Yasuaki Hiraoka.
\newblock Kernel method for persistence diagrams via kernel embedding and
  weight factor.
\newblock \emph{The Journal of Machine Learning Research}, 18\penalty0
  (1):\penalty0 6947--6987, 2017.

\bibitem[Biscio and M{\o}ller(2019)]{biscio2019accumulated}
Christophe~AN Biscio and Jesper M{\o}ller.
\newblock The accumulated persistence function, a new useful functional summary
  statistic for topological data analysis, with a view to brain artery trees
  and spatial point process applications.
\newblock \emph{Journal of Computational and Graphical Statistics}, 28\penalty0
  (3):\penalty0 671--681, 2019.

\bibitem[Berry et~al.(2020)Berry, Chen, Cisewski-Kehe, and
  Fasy]{berry2020functional}
Eric Berry, Yen-Chi Chen, Jessi Cisewski-Kehe, and Brittany~Terese Fasy.
\newblock Functional summaries of persistence diagrams.
\newblock \emph{Journal of Applied and Computational Topology}, 4\penalty0
  (2):\penalty0 211--262, 2020.

\bibitem[Hensel et~al.(2021)Hensel, Moor, and Rieck]{hensel2021survey}
Felix Hensel, Michael Moor, and Bastian Rieck.
\newblock A survey of topological machine learning methods.
\newblock \emph{Frontiers in Artificial Intelligence}, 4:\penalty0 52, 2021.

\bibitem[Cisewski et~al.(2014)Cisewski, Croft, Freeman, Genovese, Khandai,
  Ozbek, and Wasserman]{cisewski2014non}
Jessi Cisewski, Rupert~AC Croft, Peter~E Freeman, Christopher~R Genovese,
  Nishikanta Khandai, Melih Ozbek, and Larry Wasserman.
\newblock {Non-parametric 3D map of the intergalactic medium using the
  Lyman-alpha forest}.
\newblock \emph{Monthly Notices of the Royal Astronomical Society},
  440\penalty0 (3):\penalty0 2599--2609, 2014.

\bibitem[Kimura and Imai(2017)]{kimura2017quantification}
Yuki Kimura and Koji Imai.
\newblock Quantification of lss using the persistent homology in the sdss
  fields.
\newblock \emph{Advances in Space Research}, 60\penalty0 (3):\penalty0
  722--736, 2017.

\bibitem[Pranav et~al.(2017)Pranav, Edelsbrunner, Van~de Weygaert, Vegter,
  Kerber, Jones, and Wintraecken]{Pranav:2017vy}
Pratyush Pranav, Herbert Edelsbrunner, Rien Van~de Weygaert, Gert Vegter,
  Michael Kerber, Bernard~JT Jones, and Mathijs Wintraecken.
\newblock {The topology of the cosmic web in terms of persistent Betti
  numbers}.
\newblock \emph{Monthly Notices of the Royal Astronomical Society},
  465\penalty0 (4):\penalty0 4281--4310, 2017.

\bibitem[Green et~al.(2019)Green, Mintz, Xu, and Cisewski-Kehe]{Green:2019uz}
Sheridan~B Green, Abby Mintz, Xin Xu, and Jessi Cisewski-Kehe.
\newblock Topology of our cosmology with persistent homology.
\newblock \emph{CHANCE}, 32\penalty0 (3):\penalty0 6--13, 2019.

\bibitem[Pranav et~al.(2019)Pranav, Van~de Weygaert, Vegter, Jones, Adler,
  Feldbrugge, Park, Buchert, and Kerber]{Pranav:2019ws}
Pratyush Pranav, Rien Van~de Weygaert, Gert Vegter, Bernard~JT Jones, Robert~J
  Adler, Job Feldbrugge, Changbom Park, Thomas Buchert, and Michael Kerber.
\newblock {Topology and geometry of Gaussian random fields I: on Betti numbers,
  Euler characteristic, and Minkowski functionals}.
\newblock \emph{Monthly Notices of the Royal Astronomical Society},
  485\penalty0 (3):\penalty0 4167--4208, 2019.

\bibitem[Xu et~al.(2019)Xu, Cisewski-Kehe, Green, and Nagai]{Xu_2019}
X.~Xu, J.~Cisewski-Kehe, S.B. Green, and D.~Nagai.
\newblock Finding cosmic voids and filament loops using topological data
  analysis.
\newblock \emph{Astronomy and Computing}, 27:\penalty0 34–52, Apr 2019.
\newblock ISSN 2213-1337.
\newblock \doi{10.1016/j.ascom.2019.02.003}.

\bibitem[Cole et~al.(2020)Cole, Biagetti, and Shiu]{Cole:2020ul}
Alex Cole, Matteo Biagetti, and Gary Shiu.
\newblock Topological echoes of primordial physics in the universe at large
  scales.
\newblock In \emph{TDA Beyond}, 2020.

\bibitem[Wilding et~al.(2021)Wilding, Nevenzeel, van~de Weygaert, Vegter,
  Pranav, Jones, Efstathiou, and Feldbrugge]{wilding2021persistent}
Georg Wilding, Keimpe Nevenzeel, Rien van~de Weygaert, Gert Vegter, Pratyush
  Pranav, Bernard~JT Jones, Konstantinos Efstathiou, and Job Feldbrugge.
\newblock Persistent homology of the cosmic web--{I}. hierarchical topology in
  {$\Lambda$CDM} cosmologies.
\newblock \emph{Monthly Notices of the Royal Astronomical Society},
  507\penalty0 (2):\penalty0 2968--2990, 2021.

\bibitem[Cisewski-Kehe et~al.(2022)Cisewski-Kehe, Fasy, Hellwing, Lovell,
  Drozda, and Wu]{cisewski2022differentiating}
Jessi Cisewski-Kehe, Brittany~Terese Fasy, Wojciech Hellwing, Mark~R Lovell,
  Pawel Drozda, and Mike Wu.
\newblock Differentiating small-scale subhalo distributions in cdm and wdm
  models using persistent homology.
\newblock \emph{Physical Review D}, 106\penalty0 (2):\penalty0 023521, 2022.

\bibitem[Qaiser et~al.(2019)Qaiser, Tsang, Taniyama, Sakamoto, Nakane, Epstein,
  and Rajpoot]{Qaiser}
Talha Qaiser, Yee Tsang, Daiki Taniyama, Naoya Sakamoto, Kazuaki Nakane, David
  Epstein, and Nasir Rajpoot.
\newblock Fast and accurate tumor segmentation of histology images using
  persistent homology and deep convolutional features.
\newblock \emph{Medical Image Analysis}, 55, 04 2019.
\newblock \doi{10.1016/j.media.2019.03.014}.

\bibitem[Cisewski-Kehe et~al.(2023)Cisewski-Kehe, Fasy, Giriyan, and
  Quist]{cisewski2023weighted}
Jessi Cisewski-Kehe, Brittany~Terese Fasy, Dhanush Giriyan, and Eli Quist.
\newblock The weighted euler characteristic transform for image shape
  classification.
\newblock \emph{arXiv preprint arXiv:2307.13940}, 2023.

\bibitem[Glenn et~al.(2024)Glenn, Cisewski-Kehe, Zhu, and
  Bement]{glenn2024confidence}
Susan Glenn, Jessi Cisewski-Kehe, Jun Zhu, and William~M Bement.
\newblock Confidence regions for a persistence diagram of a single image with
  one or more loops.
\newblock \emph{arXiv preprint arXiv:2405.01651}, 2024.

\bibitem[Bendich et~al.(2016)Bendich, Marron, Miller, Pieloch, and
  Skwerer]{bendich2014persistent}
Paul Bendich, JS~Marron, Ezra Miller, Alex Pieloch, and Sean Skwerer.
\newblock Persistent homology analysis of brain artery trees.
\newblock \emph{The Annals of Applied Statistics}, 10\penalty0 (1):\penalty0
  198--218, 2016.

\bibitem[Lawson et~al.(2019)Lawson, Sholl, Brown, Fasy, and
  Wenk]{Lawson:2019tg}
Peter Lawson, Andrew~B Sholl, J~Quincy Brown, Brittany~Terese Fasy, and Carola
  Wenk.
\newblock Persistent homology for the quantitative evaluation of architectural
  features in prostate cancer histology.
\newblock \emph{Scientific reports}, 9\penalty0 (1):\penalty0 1--15, 2019.

\bibitem[Kramar et~al.(2013)Kramar, Goullet, Kondic, and Mischaikow]{kramar}
Miroslav Kramar, Arnaud Goullet, Lou Kondic, and K~Mischaikow.
\newblock Persistence of force networks in compressed granular media.
\newblock \emph{Physical review. E, Statistical, nonlinear, and soft matter
  physics}, 87:\penalty0 042207, 04 2013.
\newblock \doi{10.1103/PhysRevE.87.042207}.

\bibitem[Khasawneh and Munch(2016)]{khasawneh2016chatter}
Firas~A Khasawneh and Elizabeth Munch.
\newblock Chatter detection in turning using persistent homology.
\newblock \emph{Mechanical Systems and Signal Processing}, 70:\penalty0
  527--541, 2016.

\bibitem[Otter et~al.(2017)Otter, Porter, Tillmann, Grindrod, and
  Harrington]{otter2017roadmap}
Nina Otter, Mason~A Porter, Ulrike Tillmann, Peter Grindrod, and Heather~A
  Harrington.
\newblock A roadmap for the computation of persistent homology.
\newblock \emph{EPJ Data Science}, 6:\penalty0 1--38, 2017.

\bibitem[Malott et~al.(2021)Malott, Verma, Singh, and
  Wilsey]{malott2021distributed}
Nicholas~O Malott, Rishi~R Verma, Rohit~P Singh, and Philip~A Wilsey.
\newblock Distributed computation of persistent homology from partitioned big
  data.
\newblock In \emph{2021 IEEE International Conference on Cluster Computing
  (CLUSTER)}, pages 344--354. IEEE, 2021.

\bibitem[Springel(2005)]{springel2005cosmological}
Volker Springel.
\newblock The cosmological simulation code gadget-2.
\newblock \emph{Monthly notices of the royal astronomical society},
  364\penalty0 (4):\penalty0 1105--1134, 2005.

\bibitem[Hellwing et~al.(2016)Hellwing, Frenk, Cautun, Bose, Helly, Jenkins,
  Sawala, and Cytowski]{hellwing2016copernicus}
Wojciech~A Hellwing, Carlos~S Frenk, Marius Cautun, Sownak Bose, John Helly,
  Adrian Jenkins, Till Sawala, and Maciej Cytowski.
\newblock The copernicus complexio: a high-resolution view of the small-scale
  universe.
\newblock \emph{Monthly Notices of the Royal Astronomical Society},
  457\penalty0 (4):\penalty0 3492--3509, 2016.

\bibitem[Libeskind et~al.(2018)Libeskind, Van De~Weygaert, Cautun, Falck,
  Tempel, Abel, Alpaslan, Arag{\'o}n-Calvo, Forero-Romero, Gonzalez,
  et~al.]{libeskind2018tracing}
Noam~I Libeskind, Rien Van De~Weygaert, Marius Cautun, Bridget Falck, Elmo
  Tempel, Tom Abel, Mehmet Alpaslan, Miguel~A Arag{\'o}n-Calvo, Jaime~E
  Forero-Romero, Roberto Gonzalez, et~al.
\newblock Tracing the cosmic web.
\newblock \emph{Monthly Notices of the Royal Astronomical Society},
  473\penalty0 (1):\penalty0 1195--1217, 2018.

\bibitem[Liu et~al.(2018)Liu, Bird, Matilla, Hill, Haiman, Madhavacheril,
  Petri, and Spergel]{liu2018massivenus}
Jia Liu, Simeon Bird, Jos{\'e} Manuel~Zorrilla Matilla, J~Colin Hill,
  Zolt{\'a}n Haiman, Mathew~S Madhavacheril, Andrea Petri, and David~N Spergel.
\newblock Massivenus: cosmological massive neutrino simulations.
\newblock \emph{Journal of Cosmology and Astroparticle Physics}, 2018\penalty0
  (03):\penalty0 049, 2018.

\bibitem[Villaescusa-Navarro et~al.(2020)Villaescusa-Navarro, Hahn, Massara,
  Banerjee, Delgado, Ramanah, Charnock, Giusarma, Li, Allys,
  et~al.]{villaescusa2020quijote}
Francisco Villaescusa-Navarro, ChangHoon Hahn, Elena Massara, Arka Banerjee,
  Ana~Maria Delgado, Doogesh~Kodi Ramanah, Tom Charnock, Elena Giusarma, Yin
  Li, Erwan Allys, et~al.
\newblock The quijote simulations.
\newblock \emph{The Astrophysical Journal Supplement Series}, 250\penalty0
  (1):\penalty0 2, 2020.

\bibitem[Fasy et~al.(2014)Fasy, Lecci, Rinaldo, Wasserman, Balakrishnan, and
  Singh]{Fasy_2014}
Brittany~Terese Fasy, Fabrizio Lecci, Alessandro Rinaldo, Larry Wasserman,
  Sivaraman Balakrishnan, and Aarti Singh.
\newblock Confidence sets for persistence diagrams.
\newblock \emph{The Annals of Statistics}, dec 2014.
\newblock \doi{10.1214/14-aos1252}.

\bibitem[Sheehy(2013)]{Sheehy2012Linear}
Donald~R Sheehy.
\newblock Linear-size approximations to the vietoris--rips filtration.
\newblock \emph{Discrete \& Computational Geometry}, 49\penalty0 (4):\penalty0
  778--796, 2013.

\bibitem[Bauer et~al.(2014)Bauer, Kerber, and
  Reininghaus]{bauer2014distributed}
Ulrich Bauer, Michael Kerber, and Jan Reininghaus.
\newblock Distributed computation of persistent homology.
\newblock In \emph{2014 proceedings of the sixteenth workshop on algorithm
  engineering and experiments (ALENEX)}, pages 31--38. SIAM, 2014.

\bibitem[Bauer(2021)]{bauer2021ripser}
Ulrich Bauer.
\newblock Ripser: efficient computation of vietoris--rips persistence barcodes.
\newblock \emph{Journal of Applied and Computational Topology}, 5\penalty0
  (3):\penalty0 391--423, 2021.

\bibitem[Gamble and Heo(2010)]{gamble2010exploring}
Jennifer Gamble and Giseon Heo.
\newblock Exploring uses of persistent homology for statistical analysis of
  landmark-based shape data.
\newblock \emph{Journal of Multivariate Analysis}, 101\penalty0 (9):\penalty0
  2184--2199, 2010.

\bibitem[Moitra et~al.(2018)Moitra, Malott, and Wilsey]{moitra2018cluster}
Anindya Moitra, Nicholas~O Malott, and Philip~A Wilsey.
\newblock Cluster-based data reduction for persistent homology.
\newblock In \emph{2018 IEEE International Conference on Big Data (Big Data)},
  pages 327--334. IEEE, 2018.

\bibitem[Verma et~al.(2021)Verma, Malott, and Wilsey]{verma2021data}
Rishi~R Verma, Nicholas~O Malott, and Philip~A Wilsey.
\newblock Data reduction and feature isolation for computing persistent
  homology on high dimensional data.
\newblock In \emph{2021 IEEE International Conference on Big Data (Big Data)},
  pages 3860--3864. IEEE, 2021.

\bibitem[Malott and Wilsey(2019)]{malott2019fast}
Nicholas~O Malott and Philip~A Wilsey.
\newblock Fast computation of persistent homology with data reduction and data
  partitioning.
\newblock In \emph{2019 IEEE International Conference on Big Data (Big Data)},
  pages 880--889. IEEE, 2019.

\bibitem[Munkres(2000)]{munkres2000topology}
James~R Munkres.
\newblock \emph{Topology}, volume~2.
\newblock Prentice Hall Upper Saddle River, 2000.

\bibitem[Munkres(2018)]{munkres2018elements}
James~R Munkres.
\newblock \emph{Elements of algebraic topology}.
\newblock CRC press, 2018.

\bibitem[Edelsbrunner and Harer(2010)]{edelsbrunner2010computational}
Herbert Edelsbrunner and John Harer.
\newblock \emph{Computational Topology: An Introduction}.
\newblock American Mathematical Soc., 2010.

\bibitem[Mileyko et~al.(2011)Mileyko, Mukherjee, and
  Harer]{mileyko2011probability}
Yuriy Mileyko, Sayan Mukherjee, and John Harer.
\newblock Probability measures on the space of persistence diagrams.
\newblock \emph{Inverse Problems}, 27\penalty0 (12):\penalty0 124007, 2011.

\bibitem[Bubenik et~al.(2020)Bubenik, Hull, Patel, and
  Whittle]{bubenik2020persistent}
Peter Bubenik, Michael Hull, Dhruv Patel, and Benjamin Whittle.
\newblock Persistent homology detects curvature.
\newblock \emph{Inverse Problems}, 36\penalty0 (2):\penalty0 025008, 2020.

\bibitem[Turkes et~al.(2022)Turkes, Montufar, and Otter]{turkes2022onPH}
Renata Turkes, Guido Montufar, and Nina Otter.
\newblock On the effectiveness of persistent homology.
\newblock In \emph{Advances in Neural Information Processing Systems}, 2022.

\bibitem[Edelsbrunner et~al.(2008)Edelsbrunner, Harer,
  et~al.]{edelsbrunner2008persistent}
Herbert Edelsbrunner, John Harer, et~al.
\newblock Persistent homology-a survey.
\newblock \emph{Contemporary mathematics}, 453:\penalty0 257--282, 2008.

\bibitem[Morozov(2007)]{morozov2007dionysus}
Dmitriy Morozov.
\newblock Dionysus, a c++ library for computing persistent homology, 2007.

\bibitem[Fasy et~al.(2021)Fasy, Kim, Lecci, Maria, Millman, Rouvreau, and
  Kim]{fasy2021package}
Brittany~T Fasy, Jisu Kim, Fabrizio Lecci, Clement Maria, David~L Millman,
  Vincent Rouvreau, and Maintainer~Jisu Kim.
\newblock Package ‘tda’, 2021.

\bibitem[Edelsbrunner et~al.(2000)Edelsbrunner, Letscher, and
  Zomorodian]{edelsbrunner2000topological}
Herbert Edelsbrunner, David Letscher, and Afra Zomorodian.
\newblock Topological persistence and simplification.
\newblock In \emph{Proceedings 41st annual symposium on foundations of computer
  science}, pages 454--463. IEEE, 2000.

\bibitem[Zomorodian and Carlsson(2005)]{zomorodian2005computing}
Afra Zomorodian and Gunnar Carlsson.
\newblock Computing persistent homology.
\newblock \emph{Discrete \& Computational Geometry}, 33\penalty0 (2):\penalty0
  249--274, 2005.

\bibitem[Morozov(2005)]{morozov2005persistence}
Dmitriy Morozov.
\newblock Persistence algorithm takes cubic time in worst case.
\newblock \emph{BioGeometry News, Dept. Comput. Sci., Duke Univ}, 2, 2005.

\bibitem[Milosavljevi{\'c} et~al.(2011)Milosavljevi{\'c}, Morozov, and
  Skraba]{milosavljevic2011zigzag}
Nikola Milosavljevi{\'c}, Dmitriy Morozov, and Primoz Skraba.
\newblock Zigzag persistent homology in matrix multiplication time.
\newblock In \emph{Proceedings of the twenty-seventh Annual Symposium on
  Computational Geometry}, pages 216--225, 2011.

\bibitem[Chazal et~al.(2015{\natexlab{a}})Chazal, Fasy, Lecci, Michel, Rinaldo,
  and Wasserman]{chazal15}
Frederic Chazal, Brittany Fasy, Fabrizio Lecci, Bertrand Michel, Alessandro
  Rinaldo, and Larry Wasserman.
\newblock Subsampling methods for persistent homology.
\newblock In \emph{Proceedings of the 32nd International Conference on Machine
  Learning}, volume~37 of \emph{Proceedings of Machine Learning Research},
  pages 2143--2151. PMLR, 07--09 Jul 2015{\natexlab{a}}.

\bibitem[Garn et~al.(2003)Garn, Elder, Robertson, and Team]{garn2003study}
HS~Garn, JF~Elder, DM~Robertson, and Lake~Studies Team.
\newblock Why study lakes? an overview of usgs lake studies in wisconsin.
\newblock 2003.

\bibitem[Fefferman et~al.(2016)Fefferman, Mitter, and
  Narayanan]{fefferman2016testing}
Charles Fefferman, Sanjoy Mitter, and Hariharan Narayanan.
\newblock Testing the manifold hypothesis.
\newblock \emph{Journal of the American Mathematical Society}, 29\penalty0
  (4):\penalty0 983--1049, 2016.

\bibitem[Chazal et~al.(2015{\natexlab{b}})Chazal, Glisse, Labruere, and
  Michel]{JMLR:v16:chazal15a}
Frederic Chazal, Marc Glisse, Catherine Labruere, and Bertrand Michel.
\newblock Convergence rates for persistence diagram estimation in topological
  data analysis.
\newblock \emph{Journal of Machine Learning Research}, 16\penalty0
  (110):\penalty0 3603--3635, 2015{\natexlab{b}}.

\bibitem[Weissman et~al.(2003)Weissman, Ordentlich, Seroussi, Verdu, and
  Weinberger]{weissman2003inequalities}
Tsachy Weissman, Erik Ordentlich, Gadiel Seroussi, Sergio Verdu, and Marcelo~J
  Weinberger.
\newblock Inequalities for the l1 deviation of the empirical distribution.
\newblock \emph{Hewlett-Packard Labs, Tech. Rep}, 2003.

\bibitem[Qian et~al.(2020)Qian, Fruit, Pirotta, and
  Lazaric]{qian2020concentration}
Jian Qian, Ronan Fruit, Matteo Pirotta, and Alessandro Lazaric.
\newblock Concentration inequalities for multinoulli random variables.
\newblock \emph{arXiv preprint arXiv:2001.11595}, 2020.

\bibitem[Garc\'ia~Trillos et~al.(2023)Garc\'ia~Trillos, He, and
  Li]{trillos2021large}
Nicolas Garc\'ia~Trillos, Pengfei He, and Chenghui Li.
\newblock Large sample spectral analysis of graph-based multi-manifold
  clustering.
\newblock \emph{Journal of Machine Learning Research}, 24\penalty0
  (143):\penalty0 1--71, 2023.

\bibitem[Lee(2003)]{lee2003introduction}
John~M. Lee.
\newblock \emph{Introduction to Smooth Manifolds}.
\newblock Graduate Texts in Mathematics. Springer, 2003.

\bibitem[Lee(2018)]{lee2018introduction}
John~M Lee.
\newblock \emph{Introduction to Riemannian manifolds}, volume~2.
\newblock Springer, 2018.

\bibitem[Garc{\'\i}a~Trillos et~al.(2020)Garc{\'\i}a~Trillos, Gerlach, Hein,
  and Slep{\v{c}}ev]{trillos2018error}
Nicol{\'a}s Garc{\'\i}a~Trillos, Moritz Gerlach, Matthias Hein, and Dejan
  Slep{\v{c}}ev.
\newblock Error estimates for spectral convergence of the graph laplacian on
  random geometric graphs toward the laplace--beltrami operator.
\newblock \emph{Foundations of Computational Mathematics}, 20\penalty0
  (4):\penalty0 827--887, 2020.

\bibitem[Federer(1959)]{federer1959curvature}
Herbert Federer.
\newblock Curvature measures.
\newblock \emph{Transactions of the American Mathematical Society}, 93\penalty0
  (3):\penalty0 418--491, 1959.

\bibitem[Struik(1961)]{struik1961lectures}
Dirk~Jan Struik.
\newblock \emph{Lectures on classical differential geometry}.
\newblock Courier Corporation, 1961.

\bibitem[Vershynin(2010)]{vershynin2010introduction}
Roman Vershynin.
\newblock Introduction to the non-asymptotic analysis of random matrices.
\newblock \emph{arXiv preprint arXiv:1011.3027}, 2010.

\end{thebibliography}

\newpage

\appendix
\section{Introduction}

For convenience, common notation is displayed in Table~\ref{tab:notation}.

\begin{table}[H]
\centering
\begin{tabular}{|c|c|}
\hline
{\bf Description} & {\bf Notation} \\
\hline
Data & $Y_{1:n}=\{y_i\}_{i=1}^n \in \R^D$ \\
\hline
True diagram &  $\True{D}$ with $\true{m}$ features \\
\hline
Diagram partition $l$ & $\mathbf{D}^l$ with $m_l$ features \\
\hline
Representative data of $({r}_{j}^l, {b}_{j}^l, {d}_{j}^l)$ & $\widetilde{Y}_{j,l}=\{y^l_j(i)\}_{i\in[n_{j}^l]}$ \\
\hline
Euclidean distance & $\rho(x,y)$ \\
\hline
Geodesic distance & $\rho_G(x,y)$ \\
\hline
Distance between data point and boundary & $\dist(y,\mathcal B)\defeq \min_{y\in\mathcal B} \rho(y,x)$ \\
\hline
Euclidean Diameter of a set $S$ & $\diameter(S)\defeq \max_{x,y\in S} \rho(x,y)$ \\
\hline
Birth and death estimates for $l^\mathrm{th}$ partition & $\hb^l$, $\hd^l$ \\
\hline
DaC birth and death estimates & $\hb$, $\hd$ \\
\hline
$\Delta^m$ & $m^{\mathrm{th}}$ dimensional simplex \\
\hline
$[n]$ & $1,2, \ldots, n$ \\
\hline
\end{tabular}
\caption{Common notation used in technical results. }
\label{tab:notation}
\end{table}

\section{Manifold Assumptions and Data Sampling }
The theoretical properties of the proposed DaC method are investigated in the following sections. The setting where data are sampled on a single manifold is considered first, which is subsequently generalized to the setting with multiple manifolds that are sufficiently separated to avoid spurious features (i.e., spurious homology group generators). 

In this section, we focus on (i) the probability of recovering a feature, $\M$, when the data sampled on $\M$ are partitioned into $m$ sub-regions, and (ii) bounds on the resulting estimation error of the birth and death time.

Let $Y_{1:n} = \{y_i\}_{i=1}^n \in \R^{D}$ be data sampled independently from a distribution $\mu$ supported on $\M$. The distribution $\mu$ is of the form
\begin{equation}\label{eq:dmu1}
    \dd\mu = \phi\ \dd \vol_\M,
\end{equation}
for density function $\phi:\M\to \R_+$, where $\dd\vol_\M$ denotes integration with respect to the Riemannian volume form associated with $\M$. 

The following assumption is used to avoid sparse regions in the data sampling on $\M$.
\Assum*

For each sub-region $l=1,\ldots, m$, there is a corresponding persistence diagram $\D^l$ with representative data of sub-feature $j$, $({r}_{j}^l, {b}_{j}^l, {d}_{j}^l)$, denoted as $\widetilde{Y}_{j,l}=\{y^l_j(i)\}_{i\in[n_{j}^l]}$, where $n_{j}^l$ is the number of observations in partition $l$ representing sub-feature $j$.
Let $\true{b}$ and $\true{d}$ be the birth and death time from $\True{D}$ (i.e., the persistence diagram using the full data).

Furthermore, let $\M$ be a smooth, compact, connected manifold without boundary with positive curvature almost everywhere and is isomorphic to a $\mo$-sphere; see  \cite{fefferman2016testing} for details on these manifold assumptions and \cite{JMLR:v16:chazal15a} for statistical analysis of TDA methods under manifold assumptions. The \textit{surface area} of $\M$ is denoted as $a_\M$, an upper bound on the absolute value of the \textit{sectional curvature} of $\M$ is $S_\M$, the \textit{reach} of $\M$ is $R_\M$, and a lower bound on the  \textit{injectivity radius} of $\M$ is $i_\M$. Suppose $y\in\M$, let $\T_y \M$ represent the tangent plane of $\M$ at $y$.
When the curvature of $\M$ is arbitrary, $\M$ may have multiple features; for example, see the eyeglass illustration in Figure 8 of \cite{Fasy_2014}.  With $n\to\infty$, these ``extra'' features have positive birth times. 
With our curvature restrictions on $\M$, finite sampling can also result in multiple features due to noise, but with infinite sampling, $\M$ has a single feature with a birth time of zero. In Section~\ref{sec:multiple}, the case with multiple features is considered.  Throughout, we use $C>1$, $0<c<1$ to denote constants that only depend on embedding dimension $\mo, \phi$, and the geometric properties of the manifold $\mathcal{M}$, such as the injectivity radius, reach, or sectional curvature.

Assumption \ref{data} guarantees that the sample size is sufficiently large in each sub-region so that every potential-feature can be recovered as $n\to\infty$ using the proposed DaC method. 
On the other hand, if we only guarantee the sub-regions have similar sizes, a data-independent partition sampling assumption (Assumption \ref{data-independent}) is an alternative option. Given Assumption \ref{data-independent}, we need to consider appropriate concentration inequalities to quantify the number of sub-regions that can be recovered.
\begin{assumption}[Data-dependent partition sampling] \label{data}
The number of data points in partition $l$, $n_{\M^{(l)}}$, is such that $n_{\M^{(l)}} \geq n c_l \geq n c_{\min } >0$, for some constant $0< c_l < 1$ where $c_{\min}\defeq\min_{l}\{c_l\}=\Omega(1/m)$\footnote{Where $b(n)=\Omega(a(n))$ means $\lim_{n\rightarrow \infty} b(n)/a(n) = c$, for some constant $c>0$.}.
\end{assumption}
Assumption~\ref{data} is a data-dependent assumption. We can consider the following data-independent assumption: 
\begin{assumption}[Data-independent partition sampling] \label{data-independent}
For a partition $\M^{(l)}$ of $\M$, 
\begin{align}\label{eq-assum:data_alternative}
    p_l \defeq \int_{y\in \M^{(l)}}\phi(y)\ \dd\vol_\M\ge c_l\ge\Omega(1/m),
\end{align}
so that $p_l$ can be thought of as the probability of a data point appearing in sub-region $l$.
\end{assumption}

With Assumption~\ref{data-independent}, consider the following McDiarmid's inequality, which provides a probability bound on observing enough data points within each sub-region.
\begin{proposition}[McDiarmid's inequality]\label{prop:McDiarmid's inequality}
 Let $p \in \Delta^m$ and $\widehat{p}_l \sim \frac{1}{n} \operatorname{Binomial}(n, p_l)$. Then, for any $\wideeps>0$ and $l\in[m]$,
\begin{align}
    \prob\left(\left|n\widehat{p}_l-np_l\right| \geq \wideeps\right) \leq 2 \exp \left( \frac{-2\wideeps^2}{n}\right).
\end{align}
\end{proposition}

Proposition \ref{prop:McDiarmid's inequality} guarantees that with probability at least $1-2m\exp\left(\frac{-2\wideeps^2}{n}\right)$, for any $l$, $n p_l-\wideeps \leq n_{\M^{(l)}} \leq n p_l+\wideeps$. To ensure this probability close to one, then 
\begin{equation}\label{eq:eq1}
    \wideeps\gtrsim\sqrt{n}.
\end{equation}
To ensure there is a sufficient number of data points in each sub-region, we require 
\begin{equation}\label{eq:eq2}
    n p_l \gtrsim \wideeps.
\end{equation} 
If $p_l\ge \Omega(1/m)$, the Equations~\eqref{eq:eq1} and \eqref{eq:eq2} imply that  $m$ is upper bounded by $\OO(\sqrt{n})$. This scaling is not consistent with the scaling needed to achieve the linearized computational complexity for the DaC method discussed in section \ref{sec:memory}. 

 As the union bound from Proposition~\ref{prop:McDiarmid's inequality} does not have a reasonable scaling of $m$ with $n$ using Assumption~\ref{data-independent}, we consider a different approach to bound the probabilities associated with partition sampling.
 Let the number of points follow a multinomial distribution with probability parameters $p = (p_1, \ldots, p_m)$, as defined in Assumption~\ref{data-independent}, with estimates $\hat{p}_l=n_l/n$.  Then the following proposition adapted from Theorem 2.1 in \cite{weissman2003inequalities} holds: 
 \begin{proposition} \label{prop:multinomial}
 Let $p \in \Delta^m$ and $\widehat{p} \sim 1/n \operatorname{Multinomial}(n, p)$. Then, for any $m \geq 2$ and $\wideeps \in[0,1]$,
\begin{equation}\label{eq:l1bound}
\prob\left(\sum_{l=1}^m\left|p_l-\hat{p}_l\right| \geq \wideeps\right) 
\leq\left(2^m-2\right) e^{-n  \wideeps^2 / 2}.
\end{equation}
 \end{proposition}
When $p_l=\Omega(1/m)$, in order to have reasonable control on the lower bound of all $\hat{p}_l$ in Assumption~\ref{data-independent} with high probability, it requires $m\lesssim \smallO(n^{\frac{1}{3}})$. This dependency of the error probability rate on $m$ in Equation~\eqref{eq:l1bound} seems unavoidable, as discussed in \cite{qian2020concentration} about various multinomial distribution inequalities. Therefore, we need to be more careful with the rate of $m$ with respect to $n$.
Under the case that $m=\Omega(n/\log n)$, the DaC method recovers a linear computational complexity. However, a direct application of Equation~\eqref{eq:l1bound} does not guarantee a uniform bound under the data-independent Assumption~\ref{data-independent} because the error can increase substantially. Fortunately, the $l_1$ bound from Equation~\eqref{eq:l1bound} guarantees only a small number of missing sub-regions and therefore controls the error. 
For example, suppose we define the number of points within each sub-region less than $\frac{n c_{\min} }{2}$ to be \textit{unrecoverable} and greater than $\frac{n c_{\min} }{2}$ to be \textit{recoverable}. (The constant $\frac{1}{2}$ is not sharp and can be changed to any positive constant smaller than $1$.) When $\sum_{l=1}^m\left|p_l-\hat{p}_l\right|< \wideeps$, there are at most $2\wideeps/c_{\min} \approx 2m\wideeps$ unrecoverable sub-regions. Otherwise, there must be at  least $2\wideeps/c_{\min}$ sub-regions such that $\hat{p}_l< c_{\min}/2$, and thus $\sum_{l=1}^m\left|p_l-\hat{p}_l\right|>\frac{c_{\min}}{2}\frac{2\wideeps}{c_{\min}}=\wideeps$, which is a contradiction.

If $m=n/(4\log n)$ and $\wideeps=1/\sqrt{\log n}$, then 
\begin{align*}
 \prob\left(\sum_{l=1}^m\left|p_l-\hat{p}_l\right| \geq \wideeps \right) &\stackrel{\eqref{eq:l1bound}}{\leq} \exp\left(m-n\wideeps^2 /2\right)\\
    &\stackrel{m=\frac{n}{4\log n}}{=} \exp\left(\frac{n}{4\log n} -n\wideeps^2/2\right)\\
    &\stackrel{\wideeps=1/\sqrt{\log n}}{=} \exp\left(-\frac{n}{4\log n}\right).
\end{align*}
Though at a slow rate, $\wideeps$ converges to $0$, and the probability also converges to $0$. This gives a probabilistic guarantee to the case $m\sim n/\log(n)$ using the distribution-independent assumption (Assumption \ref{data-independent}). For conciseness, we generally use Assumption \ref{data} in the following proof, unless otherwise specified. 


The next assumption is about the design of the partition for the DaC to avoid the partition boundary masking parts of the manifold.  In particular, the boundary cannot be tangent to $\M$.
	\begin{assumption}[Manifold-Boundary angle]\label{assum:manifold angle}
		For $\M$ and $\B$, we assume:
		
		\begin{enumerate}
			\item The intersection $\M \cap \B$ is either the empty set or finitely many smooth, connected manifolds of dimension strictly smaller than $\mo$.
			\item For every point $y$ in $\M\cap \B$,
			\begin{equation}
			\label{eqn:AngleConstraint}
			\sup_{v\in \T_y(\M\cap\B)^{\perp_\M}, \widetilde{v}\in \T_y(\M\cap\B)^{\perp_\B}} |\angle(v,\widetilde{v})-\frac{\pi}{2}|\leq\beta,
			\end{equation}
			for some fixed $\beta$ strictly smaller than $\frac{\pi}{2}$.  In the above, $\angle(v,\widetilde{v})=\arccos \frac{\langle v, \widetilde{v}\rangle}{|v|\cdot |\widetilde{v}|}$, and $\T_y(\M\cap\B)^{\perp_\M}$ denotes the orthogonal complement of $\T_y(\M\cap\B)$ in $\T_y\M$; $\T_y(\M\cap\B)^{\perp_\B}$ is defined analogously. See an illustration in Figure~\ref{fig:assumption1}.
		\end{enumerate}
		\begin{figure}
		\centering
		\includegraphics[scale=0.5]{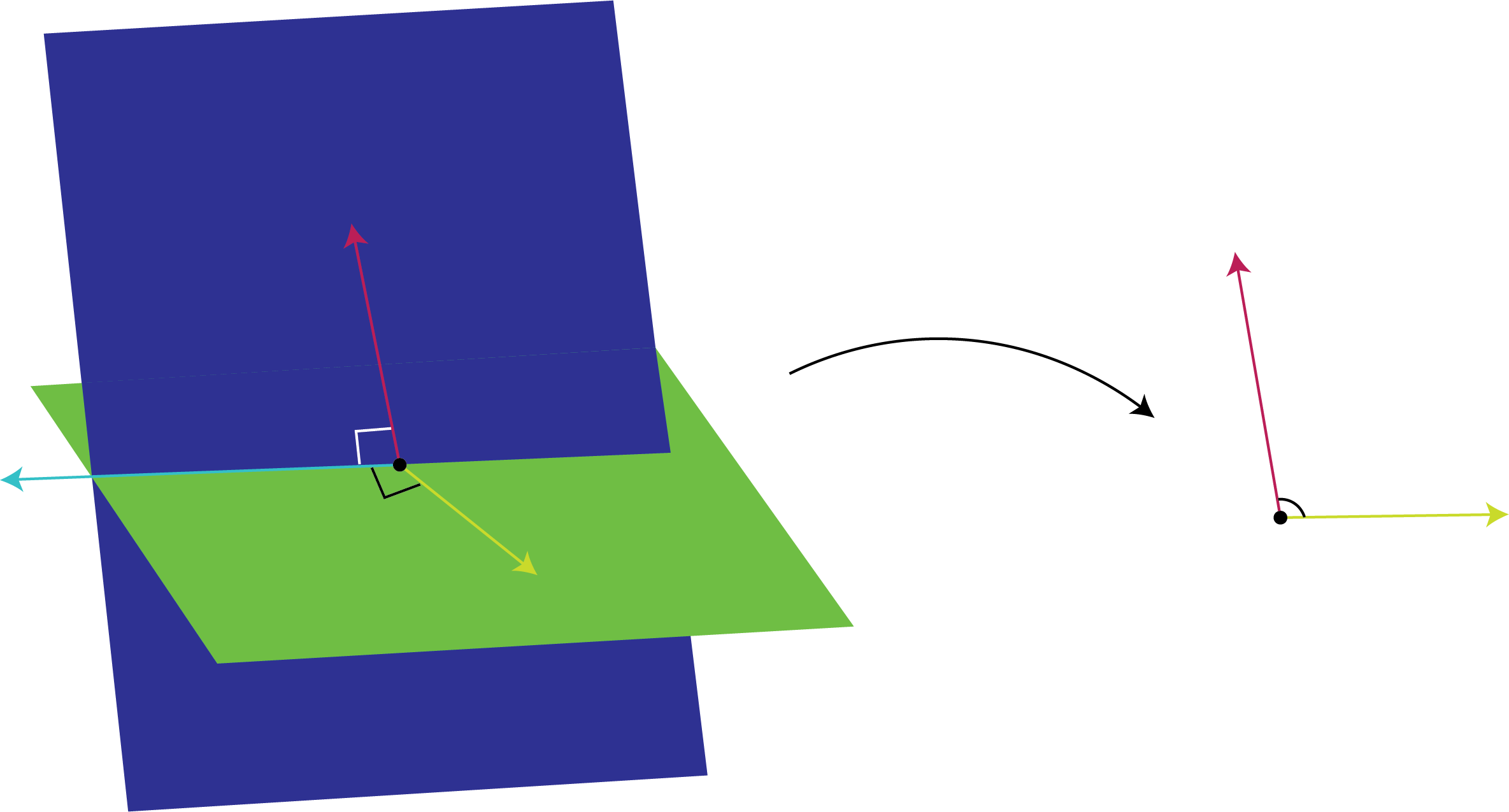}
		\put(-150,50){$\B$}
		\put(-240,170){$\T_y\M$}
		\put(12,54){$\T_y(\M\cap\B)^{\perp_\B}$}
		\put(-80,120){$\T_y(\M\cap\B)^{\perp_\M}$}
		\put(-350,60){$\T_y(\M\cap\B)$}
    \put(-200,60){$y$}
    \put(-54,48){$y$}
		\caption{Manifold-Boundary angle assumption illustration.  $\T_y(\M\cap\B)^{\perp \M}$ (red arrow) is the complement of $\T_y\M$ (blue surface) to $\T_y(\M\cap\B)$ (blue arrow), and  $\T_y(\M\cap\B)^{\perp \B}$ (green arrow) is the complement of $\B$ (green surface) to $\T_y(\M\cap\B)$ (blue arrow). We restrict the angle between $\T_y(\M\cap\B)^{\perp \M}$ and $\T_y(\M\cap\B)^{\perp \B}$ to not be too small.}
		\label{fig:assumption1}
	    \end{figure}
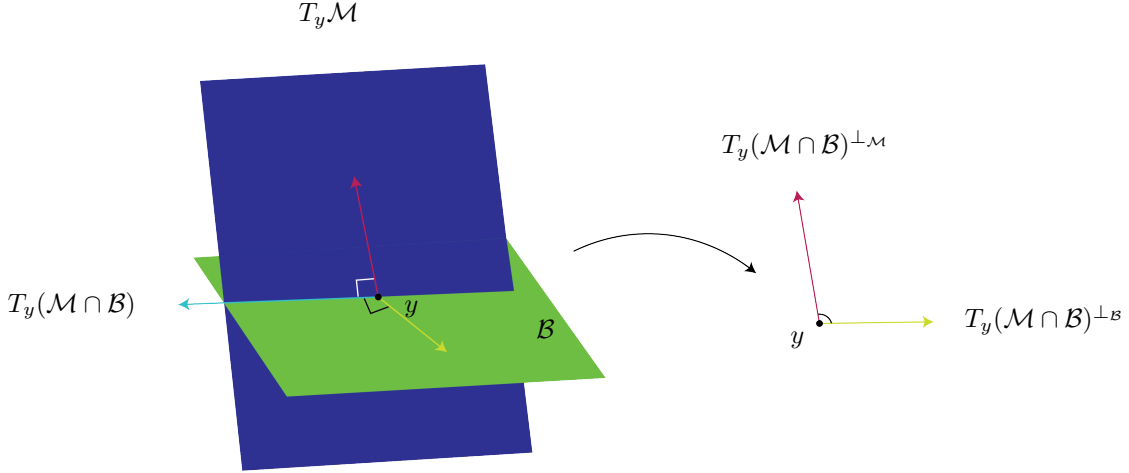
	\end{assumption}

The Manifold-Boundary angle assumption prevents the case where $\B$ is tangent to $\M$ where $\B$ can mask portions of data sampled on $\M$, which may prevent the recovery of a feature. This possibility of masking data due to the boundary tangent to the manifold is a limitation of the DaC method. However, the data partition $\B^{(l)}$ may be defined to be orthogonal to $\M^{(l)}$ to mitigate this issue.  

The following assumes the estimation error $\delta$ cannot be too large (or the sample size $n$ cannot be too small) compared to some intrinsic properties of the manifold so that locally the manifold behaves like a Euclidean space.
\begin{assumption} [Local Error Bound]\label{assum:error bound} The local error bound, $\delta>0$, satisfies
\begin{equation}
\frac{2\delta}{\cos\beta}+\frac{C\delta^2}{\cos^2\beta}\le\min\left\{1,\frac{R_\M}{2},i_\M,
 S_\M^{-\frac{1}{2}}\right\}. \label{eq:local_error_bound}
\end{equation}
\end{assumption}
The upper bound terms $\min\{i_\M,S_\M^{-\frac{1}{2}}\}$ guarantees that if $x\in\M$, then the exponential map 
\begin{equation} \label{expmap}
    \exp_{x}: \T_{x}\mathcal{M} \to \mathcal{M}
\end{equation}
is a diffeomorphism. Likewise, its inverse, the logarithm map,
\begin{equation} \label{logmap}
    \log_{x}: \mathcal{M} \to \T_{x}\mathcal{M}
\end{equation}
is a diffeomorphism. The upper bound of the reach is used to guarantee the uniqueness of the closest point of $x\in\B$ to $\M$. Assumption~\ref{assum:error bound} helps to control the difference between the geodesic and Euclidean distances, which is widely adopted in manifold learning (e.g., \citealt{trillos2021large}).

\section{Riemannian Geometry Overview}
An overview of some relevant topics of differential geometry is provided in this section. For a more comprehensive introduction, see \citet{lee2003introduction,lee2018introduction}.

Let $B_{\mathcal{M}}(x, r) \subset \mathcal{M}$ denote a geodesic ball in $\mathcal{M}$ of radius $r$ centered at $x$, and $B(x, r)$ denote a Euclidean ball in $\R^d$. For each $x \in \mathcal{M}$, $\log_x: \M \rightarrow \T_x\M$ is the Riemannian logarithmic map. For any $0<r<\min \left\{i_\M, S_\M^{-1 / 2}\right\}$, $\log _x$ is a diffeomorphism between the geodesic ball $B_{\mathcal{M}}(x, r) \subset \mathcal{M}$ and the ball $B(0, r) \subset T_x \mathcal{M}$. Conversely, the Riemannian exponential map is the inverse map of the Riemannian logarithm map. 
By Proposition 2 in \cite{trillos2018error},
\begin{align}\label{eq:d_M}
    \rho(x,y) \leq \rho_G(x, y) \leq \rho(x,y)+\frac{8}{R_\M^2}\rho^3(x,y),
\end{align}
for $x,y\in\M$ such that $\rho(x,y) \leq R_\M / 2$. This can be shown with the Rauch comparison theorem.
The reach of $\M$, denoted as $R_\M$, is the largest number such that any point at a distance less than $R_\M$ from $\M$ has a unique nearest point on $\M$  \citep{federer1959curvature}.

The following Lemma guarantees that locally, a smooth manifold can be approximated as a quadratic function; see an illustration in Figure \ref{fig:quadratic function}.
\begin{lemma} \label{lem:quadbehavior}
Let $x,z\in\M$ such that $\rho(x,z)\le \min\{1,S_{\M}^{-1/2},i_{\M}\}$, and let $\log_z(x)=x'$. Then
\begin{equation} \label{quadbehavior}
    \rho(z,x')\le C\rho^2(z,x),
\end{equation}
where constant $C>1$ is related to intrinsic properties of $\M$.
\end{lemma}
\noindent This lemma can be proved using the logarithm map's Taylor expansion from \cite{struik1961lectures}.

\begin{figure}
		\centering
		\includegraphics[scale=0.5]{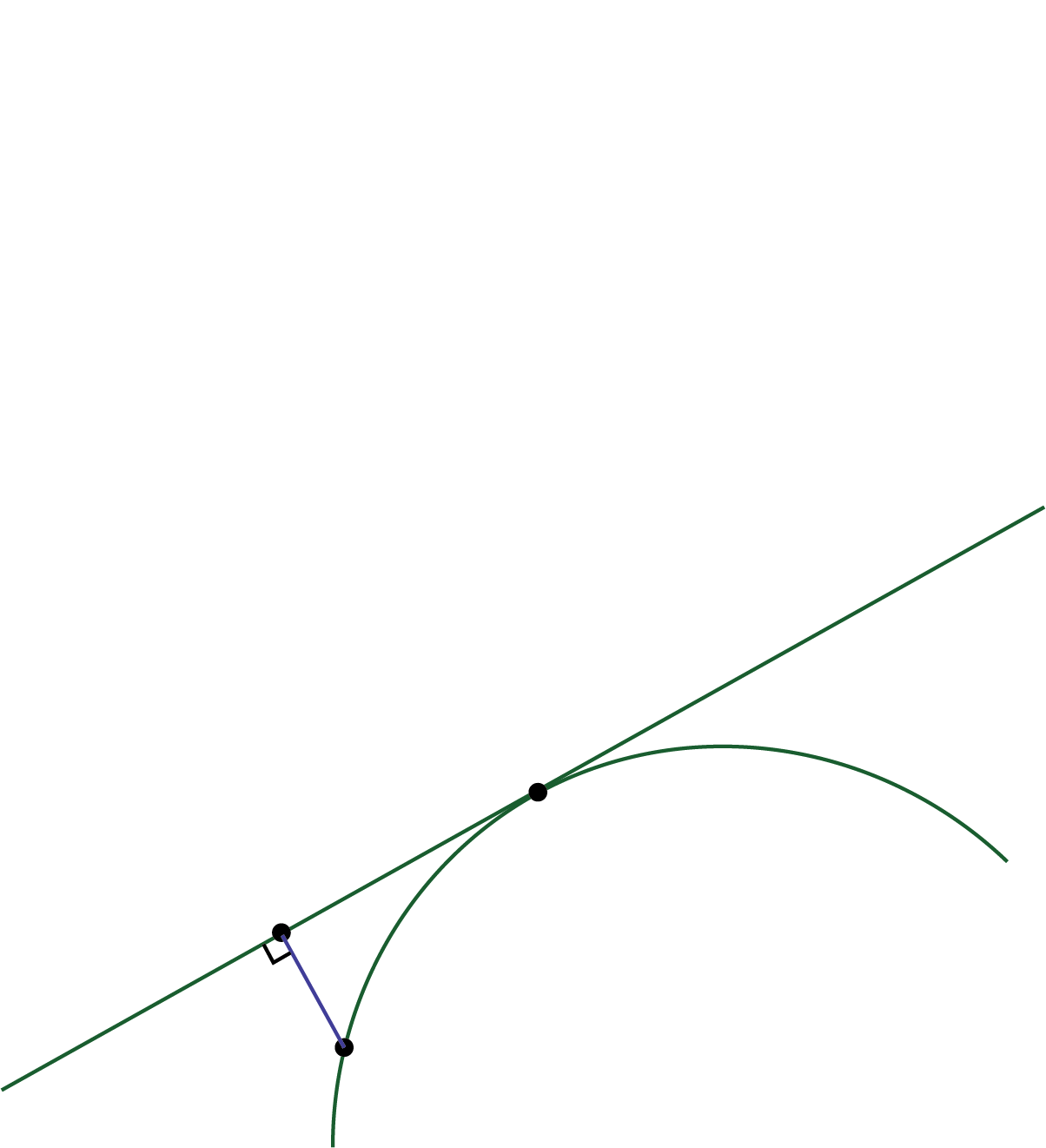}
        \put(-78,55){$x$}
	\put(-129,37){$x+\Delta x$}
        \put(-94,0){$z$}
        \put(10,40){$\M$}
        \put(0,90){$\T_x\M$}
		\caption{A manifold is locally the graph of a function $f\in C^2(\R^D)$ which can be approximated by a quadratic function around $x$, thus we have $|f(z)-f(x+\Delta x)|\approx C |\Delta x|^2$ where $C>1$ only depends on $\M$. Here we have illustrated the case where $f(x)=x$ is a scalar function.}
  \label{fig:quadratic function}
  \end{figure}
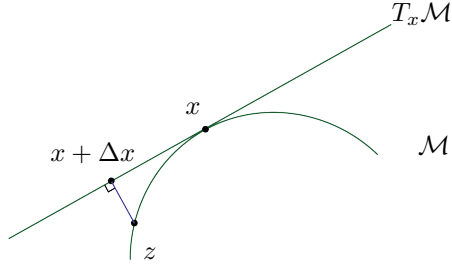

\section{Identifiability of the True Birth and Death Times}\label{sec:proof:prel}
We investigate the DaC method by considering a single feature $\M$ with its birth and death time, $b_\M$ and $d_\M$, respectively, and analyze how the estimation error in the birth and death time of the DaC persistence diagram changes as the sample size increases, $n\to \infty$.
We focus here on the case where $b_\M=0$ (e.g., $\M$ could be a loop, but not c-shaped). This is because when $b_\M>0$, the recovery of a feature with the DaC method relies on the partitioning design of $\M$; see an illustration of this issue in Figures~\ref{fig:c>0} and \ref{fig:remedyc_ge0}. Generalization to the $b_\M>0$ case requires a precise description of the partitioning of $\M$ as in Figure~\ref{fig:remedyc_ge0}. In particular, the gap in the c-shape needs to be split into distinct sub-regions as $m$ increases. This problem is similar to the multi-manifold case; see Section~\ref{sec:multiple}.

The following proposition provides bounds on $\true{b}$ and $\true{d}$, which are the birth and death times, respectively, of the feature from $\M$ when using the full dataset, $Y_{1:n}$.
\begin{sloppypar}
\begin{proposition}
\label{prop:manifold}
For error bound $\varepsilon >0$, with probability at least $1-\frac{a_\M}{ \sigma_k (\frac{\eps}{6})^{\mo} }\exp\left(-2n \sigma_k  \frac{c_\phi (\frac{\eps}{2})^{\mo}}{a_\M}\right)$, \begin{equation}
    \true{b}\le 2\eps \text{ and } |\true{d}-d_\M|\le 2\eps,
\end{equation} where
\begin{align}\label{eq:assum-eps}
    2\eps< \min\left\{1,\frac{R_\M}{2},i_\M,S_\M^{-\frac{1}{2}},\frac{d_\M-b_\M}{2}\right\},
\end{align}
and $\sigma_k$ is the volume of $k$-dimensional sphere with unit radius.
\end{proposition}
\end{sloppypar}

\begin{proof}
Note that both $\true{b}$ and $\true{d}$ are distances between two data points contributing to $\M$, where $\true{b}$ is a distance between a pair of local representative points. 
A concentration inequality is used to prove this proposition by arguing that the manifold is well represented by the sampling points. Consider a covering of $\varepsilon/2$-balls $B_{\M}(x,\eps/2)$ for all $x\in\M$. By the compactness of $\M$, there exists a finite subcover of $B_{\M}(x,\eps/2)$ that covers the manifold such that the number of subcovers is at most $\OO\left( a_\M/(\sigma_\mo(\eps/2)^\mo)\right)$ where $\sigma_\mo$ is the volume of a $\mo$-dimensional ball with a unit radius. We show that with a high probability for large enough $n$, there is at least one point falling into every covering of $\M$. In turn, the points in the covering allow for the estimation of the birth and death time of $\M$.

First, consider the $\eps/2$-ball covering $\{\M^S_i\}_{i=1, \dots, K_\M(\eps)}$ of $\M$ such that $\cup_{i=1}^{K_\M(\eps)} \M^S_i=\M$. The surface area of $\M^S_i$ is $\sigma_k(\eps/2)^{\mo}$ and the Euclidean diameter of $\M^S_i$ is no larger than $ \eps$ because the Euclidean distance is no larger than the geodesic distance. When sampling $y\sim\phi(\cdot)$ on $\M$, the probability of $y$ falling into $\M^S_i$ satisfies,
\begin{align}
    \prob(A_i)=\int_{y\in\M^S_i}\phi(y)\, \dy\ge \sigma_k c_\phi\left(\eps/2\right)^{\mo},
\end{align} 
where $A_i$ is the event that there exists some $y \in Y_{1:n}$ in $\M_i^S$.
With probability at least $ 1-(1-\sigma_k c_\phi(\eps/2)^{\mo})^{n}  $, there is at least one point in $\M^S_i$. For the complement events $A^\mathrm{C}_1,\dots, A^\mathrm{C}_{K_\M(\eps)}$, Boole's inequality shows that
\begin{align}
    \mathrm{Pr}(\cup_{i=1}^{K_\M(\eps)} A^\mathrm{C}_i)\le \sum_{i=1}^{K_\M(\eps)} \mathrm{Pr}(A^\mathrm{C}_i).
\end{align}
By applying Boole's inequality for a large $n$ and noticing the fact that $\log(1-x)\le -x$ for any $0<x<1$, with a probability of at least
\begin{equation}
    \begin{split}
        1-\frac{a_\M}{ \sigma_k (\frac{\eps}{6})^{\mo} }\left(1-\sigma_k c_\phi\left(\frac{\eps}{2}\right)^{\mo}\right)^{n} 
        &= 1-\frac{a_\M}{ \sigma_k (\frac{\eps}{6})^{\mo} } \exp\left(n\log\left(1-\sigma_k c_\phi\left(\frac{\eps}{2}\right)^{\mo}\right)\right)\\
        &\ge 1-\frac{a_\M}{ \sigma_k(\frac{\eps}{6})^{\mo} }\exp\left(-n\sigma_k c_\phi\left(\frac{\eps}{2}\right)^{\mo}\right),
    \end{split}
\end{equation}
there is at least one point falling in every $\M^S_i$ where we use the estimate of $K_\M(\eps)$ as $a_\M/(\sigma_k (\frac{\eps}{6})^\mo)$ due to counting theory; see Lemma 5.2 of \cite{vershynin2010introduction}. 
	
For arbitrary $x\in\M$, suppose $x$ is in some partition $\M_i^S$. Using the same argument as above, with high probability, there is at least one point $y_i$ of $Y_{1:n}$ in $\M_i^S$, which also implies that $\rho(x,y_i)\le\eps$. Conditioning on at least one data point from $\Yn$ in each $\M_i^S$, we have 
\begin{align}
|\true{d}- d_\M|\le 2\eps,\\
    |\true{b} - b_\M| = |\true{b} - 0| \le 2\eps.
\end{align}
	
	
	Finally, by using Equation~\eqref{eq:assum-eps}, we can see that 
\begin{equation}
\true{d}\ge d_\M-2\eps> b_\M+2\eps \ge\true{b}.
\end{equation}
Hence, with high probability for a large enough sample size $n$, the feature $\M$ is detectable with persistent homology.
\end{proof}

With high probability, if we sample enough data points from $\M$, we can recover the birth and death time of $\M$ using the full dataset. In the following, we assess the consistency of the DaC method to recover the feature associated with $\M$, given Proposition~\ref{prop:manifold}.

\section{Consistency of the DaC persistence diagram}\label{appen:DaC}
In this section, we study the consistency of the proposed DaC method, in which we focus on the merging of the sub-features to recover $\M$. We use the estimated persistence to assess if a feature is recovered (i.e., if estimated death time $>$ estimated birth time), rather than the values of the estimates. We study probabilistic guarantees of the divide step and merging step of DaC, and quantify the possible induced error.
Recall $\M$ is the $\mo$-dimensional feature to be recovered, as defined in Section~\ref{sec:proof:prel}. Assume $\M$ is divided into $m$ partitions, with boundary $\B$ forming the partition. Assume $\B$ is of zero curvature almost everywhere (i.e., a union of linear subspaces of dimension at least $\mo$ in $\R^D$).  
The partition boundaries of $\M^{(l)}$ with respect to $\B$ are denoted as $\mathcal{B}^{(l)}$. 
Since $\M$ is assumed to possess non-negative curvature, there are sections of boundary $\Bar{\mathcal B}^{(l)}\subseteq\mathcal B^{(l)}$ such that $\M^{(l)}\cup \Bar{\mathcal{B}}^{(l)}$ (i.e., the sub-feature with its boundary) is also a manifold with non-negative curvature almost everywhere. In particular, $\Bar{\mathcal B}^{(l)}$ are the parts of the boundary from $\B$ that contribute to forming the sub-feature.
Denote the death time of $\M^{(l)}\cup\Bar{\mathcal B}^{(l)}$ as $d_{\M^{(l)}\cup\bar{\B}^{(l)}}$. According to the assumption that the birth time of $\M$ is $0$, the birth time of $\M^{(l)}\cup \Bar{\mathcal B}^{(l)}$ is also $0$. The proposed DaC algorithm produces an estimated birth and death time of $(\Yn\cap\M^{(l)})\cup\Bar{\mathcal B}^{(l)}$ where we denote the birth and death time estimates as $\hb^l$ and $\hd^l$, respectively. Recall that $(\hb^l,\hd^l)$ is called the sub-region-$l$ potential diagram.

\subsection{Identifiability of potential features}
Recall that $\M^{(l)}\cup\Bar{\mathcal B}^{(l)}$ is the union of the part of the manifold in sub-region $l$ with its boundary, which forms a non-negative curvature manifold. The data point $y_i$ is sampled according to distribution $\phi$ on $\M^{(l)}$, which satisfies Assumption~\ref{data} and Equation~\eqref{eq:dmu1}. The proof begins with the recovery guarantees for individual sub-regions $l$, containing $\M^{(l)}\cup\Bar{\mathcal B}^{(l)}$. When the number of data points sampled within $\M^{(l)}$ is sufficiently large, with high probability $\M^{(l)}\cup\Bar{\mathcal B}^{(l)}$ can be recovered by $(\Yn\cap\M^{(l)})\cup\Bar{\mathcal B}^{(l)}$. This also implies that DaC recovers the birth and death time of $\M^{(l)}\cup\Bar{\mathcal B}^{(l)}$ with a small error. We quantify the error that is introduced when considering the merge step. In particular, given that most representative points can be recovered using DaC, the DaC-based estimates are close to the true estimates (when using the full dataset). 

Before presenting our main result, we need an assumption to ensure the error bound $\delta$ is not too large compared with the sub-features.
\begin{assumption}[Global Error bound]\label{assum:error bound for delta: individual M}
The bound on the error $\delta$ due to global properties of $\M$ is $0 < \delta < \min_{l\in[m]}\left\{\frac{d_{\M^{(l)}\cup\bar{\B}^{(l)}}-b_{\M^{(l)}\cup\bar{\B}^{(l)}} }{4}\right\}$.
\end{assumption}
Assumption~\ref{assum:error bound for delta: individual M} is used to guarantee that the error bound $\delta$ of DaC is small so that the individual $\M^{(l)}\cup\bar{\B}^{(l)}$ can be recovered. Under the assumption that the sub-regions are equal-sized hyperrectangles (and $\M$ is partitioned approximately equally into these sub-regions), we have $\min_{l\in [m]}\left\{\frac{d_{\M^{(l)}\cup\bar{\B}^{(l)}}-b_{\M^{(l)}\cup\bar{\B}^{(l)}} }{4}\right\}\sim m^{-\frac{1}{k}}$. This induces an upper bound assumption on the estimation error.

The following lemma shows that DaC recovers $\M^{(l)}\cup\Bar{\mathcal B}^{(l)}$ with high probability.
\begin{lemma}\label{lem:sub-features can be recovered}
	Given Assumptions~\ref{data}, \ref{assum:manifold angle}, \ref{assum:error bound} and \ref{assum:error bound for delta: individual M}, then for sampling over $\M^{(l)}$, with probability at least $ 1-\frac{a_{\M^{(l)}}}{\sigma_k(\frac{\eps}{6})^{\mo}}\exp\left(-\frac{\sigma_k nc_{\min}c_\phi(\frac{\eps}{2})^{\mo}}{ a_{\M^{(l)}} }\right)$, 
	\begin{align}\label{eq:b<delta<d_M<d}
	    \hb^l\le2\eps < d_{\M^{(l)}\cup \bar{\B}^{(l)}}-2\eps\le\hd^l,
	\end{align}
	for all $l$ where $\hb^l$ and $\hd^l$ are the DaC birth and death time estimates of $(\Yn\cap\M^{(l)})\cup \bar{\B}^{(l)}$.
\end{lemma}
\begin{proof}

    Since $\Bar{\mathcal B}^{(l)}\cup\M^{(l)}$ is a manifold with non-negative curvature, the user-selected $\eps$ satisfies Assumption~\ref{assum:error bound}, and the number of data points on $\M^{(l)}$ is sufficiently large according to Assumption~\ref{data}, we can apply Proposition~\ref{prop:manifold} to the sampling on $\Bar{\mathcal B}^{(l)}\cup\M^{(l)}$. Specifically, this follows since $\Yn\cap \M^{(l)}$ follows distribution $\phi(\cdot)$ sampled on $\M^{(l)}$, and the number of samples $n_{\M^{(l)}}$ is larger than $nc_{\min}$ according to Assumption~\ref{data}. $\bar{\B}^{(l)}$ is finitely path connected by an application of Sard's theorem \citep{lee2003introduction} because $\M^{(l)}$ is a smooth, connected, and compact manifold. 
    Therefore we may consider disconnected subspaces $\bar{\B}^{(l)}$ as supplemental data points merged with $\Yn\cap \M^{(l)}$ for Proposition~\ref{prop:manifold}. In other words, the $\bar{\B}^{(l)}$ adds a finite number of data points to the partition, and this lemma follows from Proposition~\ref{prop:manifold}.
\end{proof}

We remark that $b_{\M^{(l)}\cup \bar{\B}^{(l)}}$ is not always smaller than $d_{\M^{(l)}\cup \bar{\B}^{(l)}}$ for {\it any} partition $\bar{\B}^{(l)}$ of an {\it arbitrary} $\M$, but the assumption that $\M$ is  a compact manifold without boundary implies $0=b_{\M^{(l)}\cup \bar{\B}^{(l)}}<d_{\M^{(l)}\cup \bar{\B}^{(l)}}$; 
Figures~\ref{fig:c>0} and \ref{fig:negative curvatured} provide examples of a non-compact manifold with boundary points and a manifold with negative curvature, respectively. In the case of a positive birth time, a suitable partition is needed to satisfy the assumptions, such as the partition defined in Figure~\ref{fig:remedyc_ge0}. On the other hand, the positive curvature manifold with a birth time of $0$ does not require a special partition design.


\begin{figure}
\centering
\begin{subfigure}{0.32\textwidth} \centering
    \includegraphics[width=3.5cm]{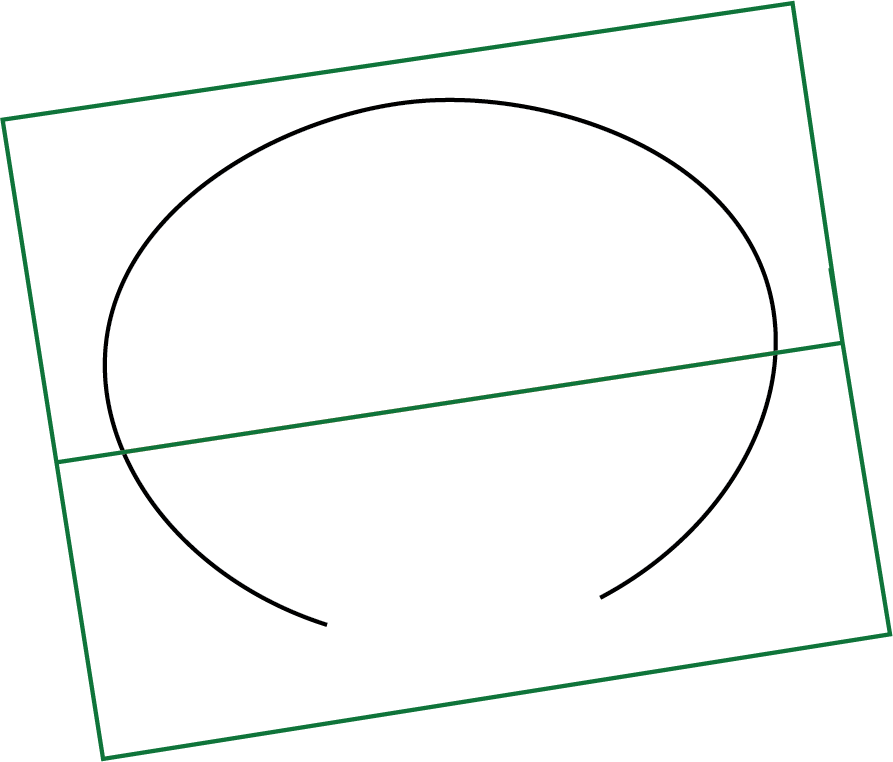}
    \put(-20,25){$\M_1$}
    \put(-97,58){$\M_2$}
    \put(-105,3){$\B_1$}
    \put(-44,83){$\B_2$}
   \caption{} \label{fig:c>0}
\end{subfigure}
\hfill
\begin{subfigure}{0.32\textwidth} \centering
    \includegraphics[width=3.5cm]{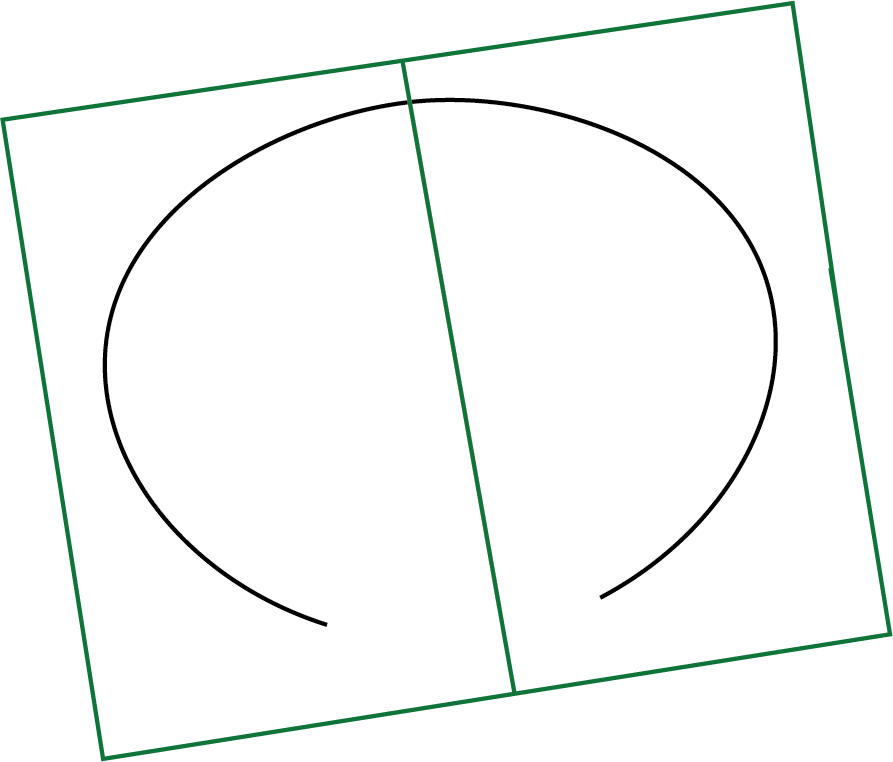}
    \put(-35,55){$\M_1$}
    \put(-80,30){$\M_2$}
    \put(-105,3){$\B_2$}
    \put(-5,65){$\B_1$}
   \caption{} \label{fig:remedyc_ge0}
\end{subfigure}
\hfill
\begin{subfigure}{0.32\textwidth} \centering
    \includegraphics[width=3.5cm]{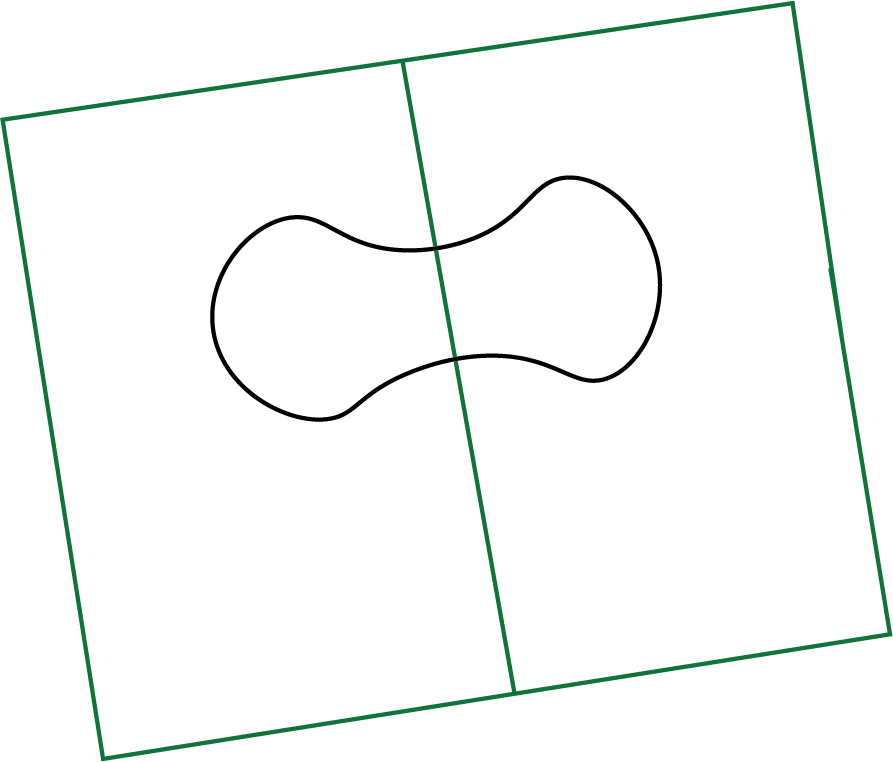}
    \put(-25,35){$\M_1$}
    \put(-90,30){$\M_2$}
    \put(-105,3){$\B_2$}
    \put(-5,65){$\B_1$}
   \caption{}
			\label{fig:negative curvatured}
\end{subfigure}
  \caption{The sub-space boundary $\B=\B_1\cup \B_2$ divides $\M$ into $\M_1$ and $\M_2$ within different rectangles.  (a) An example where $b_{\M}>0$, and $\M_1\cup\bar{\B}_1$ may not form a detectable loop because $\M$ is a non-compact manifold with boundary points. 
  One may consider a partition as in (b) to remedy this. (c) An example where $\M$ has negative curvature, which can be seen as two manifolds with birth times larger than $0$; negative curvature may induce ambiguity of the number of features.}
\label{fig:manifold_examples}
\end{figure}

\subsection{Masking effect}
Since $\B$ is a subspace rather than a single point, data in its proximity can be masked. In the following, we quantify the missing area of data points of DaC in each sub-region due to the \textit{masking effect} of $\B$, which refers to the possibility that points on $\M^{(l)}$ can be masked by $\B^{(l)} $ for $l\in [m]$. Assume only data points on $\M$ within the Euclidean distance $2\delta$ to $\B$ may be affected, for constant $\delta>0$. Namely, consider $y\in\M$ and $\Bar{y}$ to be the closest point to $y$ in $\B$, then $y$ may not be identified as a masked representative data point when $\dist(y,\B)=\rho(y,\Bar{y}) \le 2\delta\le \hat{b}_{(\Yn\cap\M^{(l)})\cup \B^{(l)}}$. To determine the masking range on $\M$, we need to go from the Euclidean distance range to the geodesic distance range on $\M$, which is quantified in Lemma \ref{lem:masked points on Mi}. Figure~\ref{fig:Masking_DEF} provides an illustration of the masking effect notation.
 \begin{figure}
		\centering
    \includegraphics[scale=0.8]{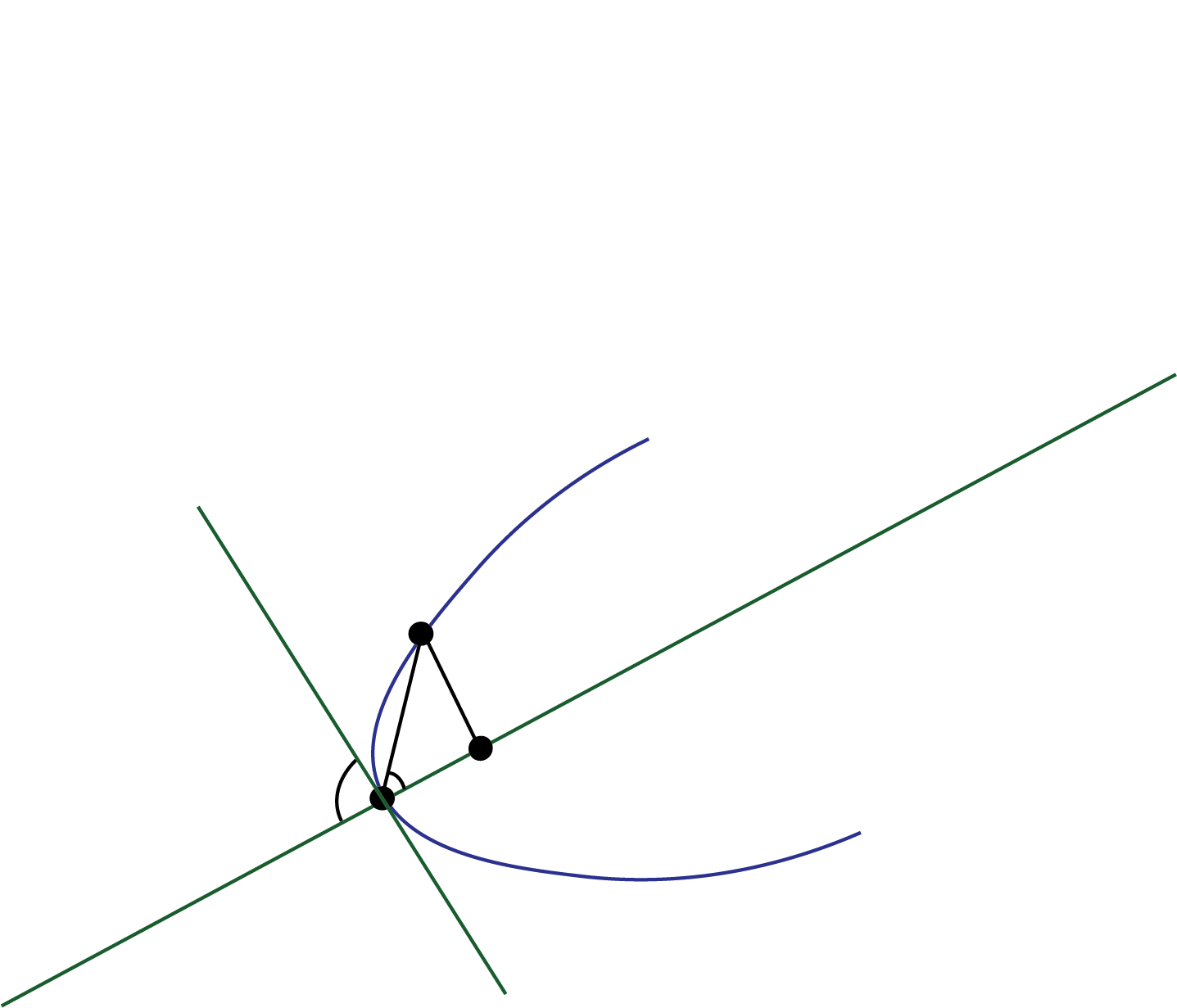}
    \put(-142,137){$\M^{(l)}$}
    \put(-60,128){$\B^{(l)}$}
    \put(-194,35){$x$}
    \put(-208,50){$\theta$}
    \put(-187,90){$y$}
    \put(-160,50){$\Bar{y}$}
    \put(-260,90){$\T_x \M^{(l)}$}
		\caption{An example of the Masking Effect of sub-region boundary $\B^{(l)}$ to $\M^{(l)}$. The point $y$ is locally masked by $\B^{(l)}$ and is not recovered as a representative point by DaC when $\rho(y, \bar{y}) \leq \hat{b}_{(\Yn\cap\M^{(l)})\cup \B^{(l)}}$. 
  Denote $\Bar{y}$ as the closest point of $y$ to $\B^{(l)}$. The possible range of masked points $y$ on $\M^{(l)}$ depends on both $\rho(y,\Bar{y})$ and the intersection angle, $\theta$ (where $|\theta -\frac{\pi}{2}|\le\beta$ according to Assumption \ref{assum:manifold angle}), between $\T_{x}\B^{(l)}$ and $\T_{x}\M^{(l)}$ where $x\in \M^{(l)}\cap\B^{(l)}$. }
		\label{fig:Masking_DEF}
	    \end{figure}
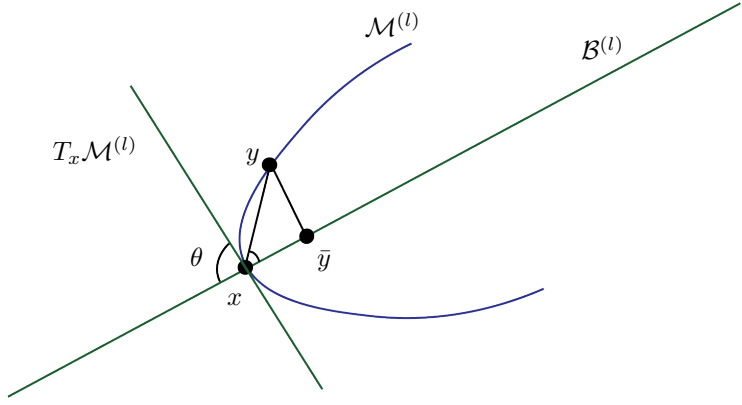
     
The following lemma quantifies the masking range.
\begin{lemma}[Masking Effect]\label{lem:masked points on Mi}
    Given Lemma~\ref{lem:sub-features can be recovered},  when $\delta\ll 1$, the representative points on $\M^{(l)}$ within a geodesic distance to $\B$ of at most $\frac{2\delta}{\cos\beta}+\frac{C\delta^2}{\cos^2\beta}$ are masked.
\end{lemma}
\begin{proof}
   Due to the masking effect at the boundary, the only potential missing representative points from the merge step to the original data are data points close to the boundary. 
    Consider $y$ as a missing representative point on $\M^{(l)}$ due to the masking effect. According to the definition of the masking effect, 
    \begin{align}\label{eq:dist(y,Bl)}
        \dist(y,\B^{(l)})\le 2\delta.
    \end{align}
    Suppose $x$ is the closest point on $\M^{(l)}\cap\B^{(l)}$ from $y$. We discuss $\rho(x,y)$ in the following two cases.
    \begin{enumerate}
        \item Assume for contradiction that $\rho(x,y)> 2\delta/\cos\beta+C\delta^2/\cos^2\beta$.
    In this case, since $\M$ is connected, there must exist $z\in\M^{(l)}$ such that 
    \begin{align}\label{eq:rho(z,x)}
        \rho(z,x)=2\delta/\cos\beta + C\delta^2/\cos^2\beta \le R_\M,
    \end{align}
    according to Assumption \ref{assum:error bound}. Also, 
    \begin{align}\label{eq:dist(z,Bl)}
        \dist(z,\B^{(l)})<\dist(y,\B^{(l)})\le 2\delta,
    \end{align}
    which follows from Equation~\eqref{eq:dist(y,Bl)}. Denote the Riemannian logarithm of $z$ on $\T_x\M$ as $z'$; the existence of the Riemannian logarithm map of $z$ is guaranteed by Assumption \ref{assum:error bound}. Therefore, we have 
    \begin{align}\label{eq:rho and rho^2}
        \rho(z,z')\le C\rho^2(x,z')=C\rho^2_G(x,z)\le C\delta^2/\cos^2\beta,
    \end{align}
    where the first inequality follows from Lemma \ref{lem:quadbehavior} and the second inequality follows from Equations \eqref{eq:d_M} and \eqref{eq:rho(z,x)}. According to Equation~\eqref{eq:d_M}, we have
    \begin{align}
        \rho(x,z')= \rho_G(x,z)\le \rho(x,z)+C\rho^3(x,z)\le 2\delta/\cos\beta+C\delta^2/\cos^2\beta,
    \end{align}
    where we omit the small order term in the last inequality. According to Assumption \ref{assum:manifold angle}, we have 
    \begin{align}\label{eq:z'Bl}
        \dist(z',\B^{(l)})\le \dist(z,\B^{(l)})+\rho(z,z')< 2\delta +C\delta^2/\cos^2\beta,
    \end{align}
    where the last inequality follows from Equations~\eqref{eq:rho and rho^2} and \eqref{eq:dist(z,Bl)}. According to Assumption \ref{assum:error bound}, we know that
    \begin{equation}
        \begin{split}
            \rho(z,x)
            &\stackrel{\text{Triangle Inequality}}{\le}\rho(z,z')+\rho(z',x)\\
            &\stackrel{\text{Angle Constraint in Assumption \ref{assum:manifold angle}}}{\le}\rho(z,z')+\frac{\dist(z',\B^{(l)})}{\cos\beta}\\
            &\stackrel{\text{Equations \eqref{eq:rho and rho^2}, \eqref{eq:z'Bl}}}{<} \frac{C\delta^2}{\cos^2\beta} + \frac{2\delta}{\cos\beta},
        \end{split}
    \end{equation}
    which contradicts the assumption that $\rho(z,x)=\frac{2\delta}{\cos\beta}+\frac{C\delta^2}{\cos^2\beta}$. This means that we must not have $\rho(x,y)> \frac{2\delta}{\cos\beta}+\frac{C\delta^2}{\cos^2\beta}$.
    
    \item Otherwise, we have $\rho(x,y)\le 2\delta/\cos\beta+C\delta^2/\cos^2\beta$. The $x$ is the unique closest point to $y$ on $\M^{(l)}\cap \B^{(l)}$ provided that $\delta$ is smaller than $R_\M$. 
    Assume $y'$ is the logarithm map from $y$ to $\T_x\M^{(l)}$; the exponential map is well defined because of Assumption \ref{assum:error bound}. See an illustration in Figure~\ref{fig:Masking}.
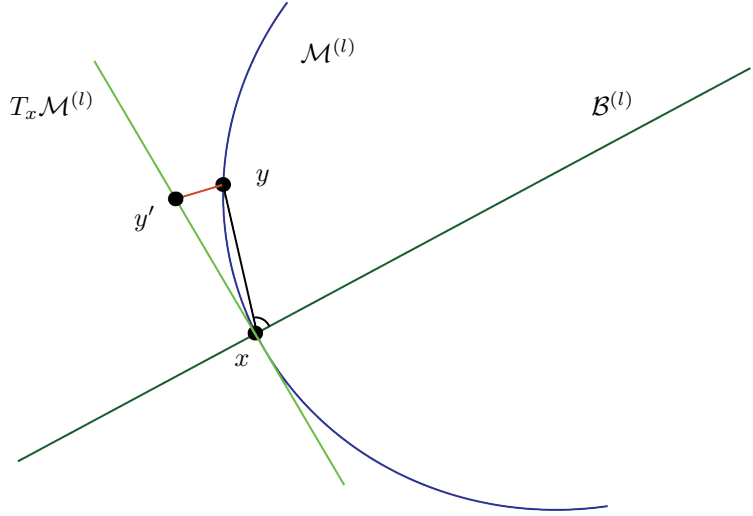
\begin{figure}
		\centering
    \includegraphics[scale=0.8]{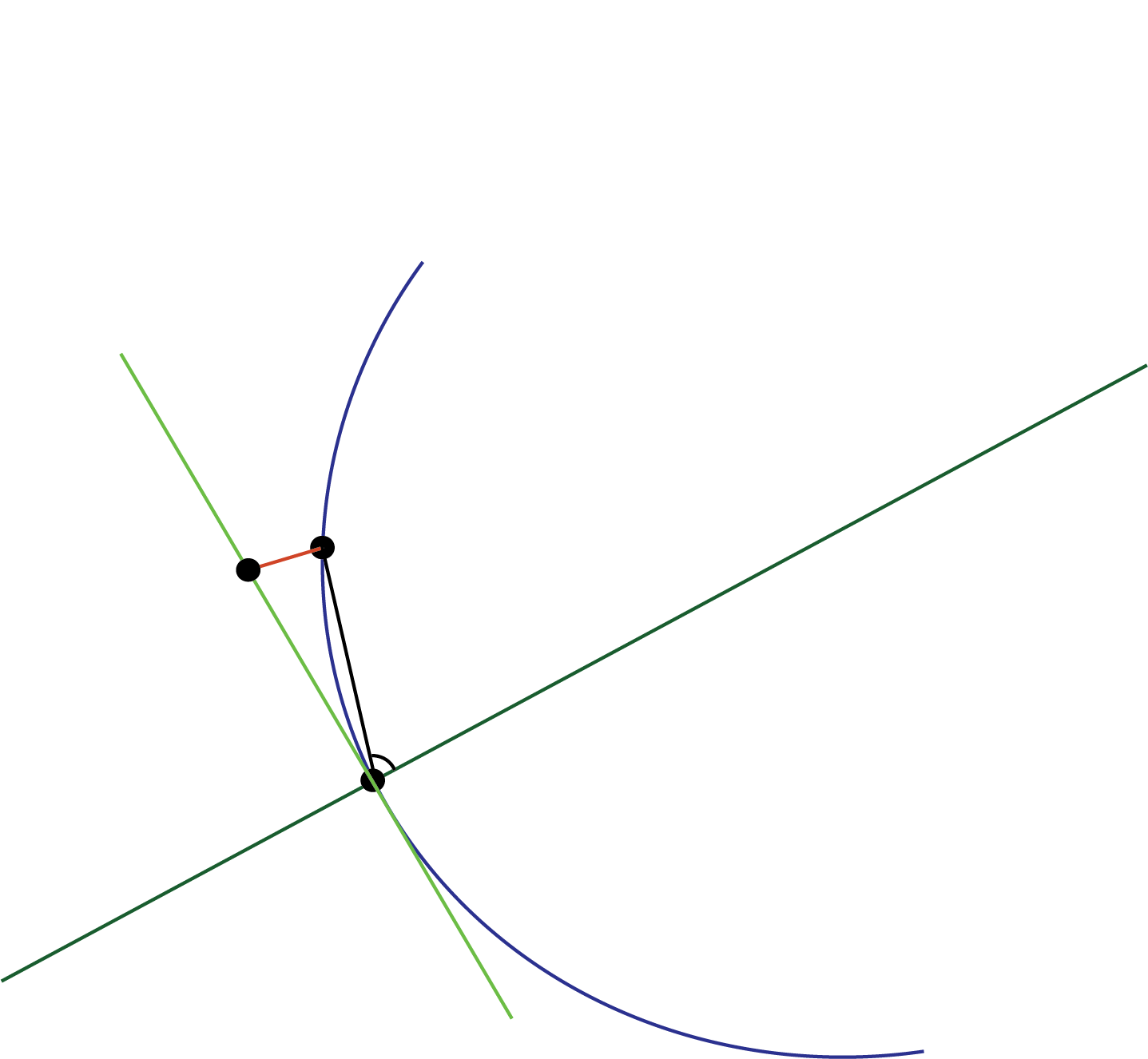}
    \put(-170,170){$\M^{(l)}$}
    \put(-60,150){$\B^{(l)}$}
    \put(-233,108){$y'$}
    \put(-280,150){$\T_x \M^{(l)}$}
    \put(-195,55){$x$}
    \put(-187,125){$y$}
		\caption{Proof illustration of Lemma~\ref{lem:masked points on Mi}.}
		\label{fig:Masking}
	    \end{figure}

From Assumption~\ref{assum:error bound}, we can see that $\rho(x,y)\le 2\delta/\cos\beta+C\delta^2/\cos^2\beta \le R_\M$. This guarantees that the masking area on $\M^{(l)}$ due to $\B^{(l
)}$ can include points of geodesic distance to $\B^{(l)}$ at most 
$2\delta/\cos\beta+C \delta^2/\cos^2\beta$
 because of Equation~\eqref{eq:d_M}. 
\end{enumerate}
\end{proof}

Lemma~\ref{lem:masked points on Mi} guarantees the masking area on $\M$ is small. Under the Assumptions~\ref{data}, we see that there are enough points within individual sub-regions, which implies that there are enough data points to estimate the birth and death time of $\M^{(l)}\cup\bar{\B}_l$.  This, in turn, guarantees the upper bound $2\delta$ of the masking effect, $\hat{b}^l$, is small. 


\subsection{Feature recovery rates for merge methods}\label{sec-app:merge method}
Given that each sub-feature can be recovered with high probability, we establish theoretical guarantees for the two merge-step methods, the Projected Merge method and the Representative Rips Merge method. In particular, we provide error-bound control for both the Projected Merge method (for when $m$ is small) and the Representative Rips Merge method (for when $m$ is large).

First, for the Projected Merge method, recall that the algorithm can cancel a set of topological projections if, under the module-2 sum, the projection error is smaller than some threshold. We show that if we set the threshold to a value related to $\delta$, the Projected Merge method is guaranteed to recover the potential-features.
\begin{lemma}[Projected Merge method]\label{lem:Projected Cancellation Method}
    Conditioned on Equation \eqref{eq:b<delta<d_M<d} in Lemma~\ref{lem:sub-features can be recovered} and given Assumptions~\ref{assum:manifold angle}, \ref{assum:error bound} and \ref{assum:error bound for delta: individual M}, when the allowed projection error to cancel the projection is set to $4\delta+C\delta^2/\cos\beta$, then with probability at least $ 1-\sum_{l=1}^m \frac{a_{\M^{(l)}}}{\sigma_k(\frac{\delta}{6})^{\mo}}\exp\left(-\frac{ n \sigma_k c_{\min} c_\phi(\frac{\delta}{2})^{\mo}}{ a_{\M^{(l)}} }\right) $, the Projected Merge method recovers the feature.
\end{lemma}
\begin{proof}
    This lemma directly follows from Lemma~\ref{lem:masked points on Mi} applied to $(\Yn\cap\M^{(l)})\cup \B^{(l)}$. Since the Euclidean distance range of the masked region is no larger than $2\delta/\cos\beta+C\delta^2/\cos^2\beta$, the projection mismatching error of data points in two contiguous sub-regions $l_1$ and $l_2$ with $Y_{1:n}\cap\M^{(l_1)}$ and $Y_{1:n}\cap\M^{(l_2)}$ is at most $4\delta+C\delta^2/\cos\beta$. The error induced by the missing data points is projected onto the shared boundary $\B$. Notice that the projected Euclidean distance to a subspace is smaller than the  Euclidean distance between the (unprojected) data points. This is the allowed projection error.
\end{proof}

The Projected Merge method is suggested for use when $m$ is small. The cancellation is also computationally cheap when $m$ is small. 
In practice, when $m$ is large, the Projected Merge method does not tolerate missing sub-regions (i.e., sub-regions that should have sub-features, but are not detected), so an alternative approach is proposed.
When $m$ is large,  the Representative Rips Merge method is suggested because it can handle the setting with some unrecovered sub-regions. However, the Representative Rips Merge method requires $m$ to be large enough to recover the feature.  For example, when searching for an $H_1$ feature in $\R^2$, $m$ must be at least four. 

Since the Representative Rips Merge method considers each sub-feature as a single point in the distance matrix, 
if a sub-region is missing (i.e., it should have a sub-feature, but it is not detected), its Euclidean diameter can be used to quantify its contribution to the error on $d_\M$ and $b_\M$ when part of $\M$ is not recovered. 
One missing sub-region would at most decrease $d_\M$ by $2C/(m^{1/\mo})$ and increase $b_\M$ by $2C/(m^{1/\mo})$, and introduce a similar effect to the DaC birth and death time estimates, $\hat{b}$ and $\hat{d}$. 
The following assumption provides an upper-bound on the number of missing sub-regions for the Representative Rips Merge method so that $\M$ is still recoverable.
\begin{assumption}[Representative Rips Merge method]\label{assum:RRM}
Suppose there are at most $r$ unrecovered sub-regions, then it must hold that $\frac{\delta}{\cos\beta}+C\frac{\delta^2}{\cos^2\beta}+C(\frac{r}{m})^{\frac{1}{\mo}} < \frac{d_\M-b_\M}{4}$, where $C$ depends on $\phi$ and the intrinsic properties of $\M$.
\end{assumption}
This assumption guarantees the global formation of $\M$ identified on the resulting persistence diagram, even if the local information of some sub-regions is missing. Recall that the number of unrecoverable sub-features can be quantified by considering the discussion below Proposition~\ref{prop:multinomial}.


\begin{lemma}[Representative Rips Merge method]\label{lem:merge-Representative Rips Merge method}
    Conditioning on Proposition \ref{prop:manifold}, and given Assumption \ref{assum:error bound} and \ref{assum:RRM}, the Representative Rips Merge method can recover the global feature $\M$.
\end{lemma}
\begin{proof}
    Recall that $\M$ is separated into $m$ similar-sized sub-regions such that each sub-region's Euclidean diameter is $C/(m^{1/\mo})$, where $C$ depends on $\phi$ and intrinsic properties of $\M$. 
    For any $x\in \M$, suppose $x\in \M^{(l)}$ where $\M^{(l)}$ is a sub-region from the proposed DaC method. Conditioned on recoverability of the sub-features $\M^{(l)}$, there exists a $y_i\in\M^{(l)} \cap \Yn$ such that $\rho(y_i,x)\le \delta/\cos\beta+C\delta^2/\cos^2\beta$; this follows a similar argument as the $\delta$-ball covering considered in the proof of Proposition \ref{prop:manifold}. 
    
    On the other hand, if the number of unrecovered sub-regions is smaller than $\mathcal{O}(m\widedelta)$, then the birth and death time estimates' errors are smaller than $2\widedelta^{1/\mo}$, where $\widedelta$ is the proportion of missing sub-regions. This is because for any $x\in\M$, there is at least one sub-region $\M^{(l)}$ that can be recovered with a point $y_i\in\M^{(l)}\cap\Yn$ such that $\rho(y_i,x)\le \widedelta^{1/\mo}$. If $m\widedelta=r$, then the error bound $2\widedelta^{1/\mo}=2\left(r/m\right)^{1/\mo}$. 
    Since we construct the distance matrix among those sub-features, both the Euclidean diameter of the covering balls, $\delta/\cos\beta+C\delta^2/\cos^2\beta$, used in Proposition \ref{prop:manifold} and the Euclidean diameter of the missing sub-regions, $\OO\left(\left(r/m\right)^{1/\mo}\right)$, appear in the birth and death time estimates' errors. By using the triangle inequality, we have
    \begin{equation}\label{eq:d^r-b^r>0}
        \hat{d}-\hat{b} \ge d_\M-b_\M-4 \left(C\left(\frac{r}{m}\right)^{\frac{1}{\mo}}+\frac{\delta}{\cos\beta}+\frac{C\delta^2}{\cos^2\beta}\right) > 0,
    \end{equation}
    where the first inequality follows from each sub-region containing at least one representative point and the second inequality follows from Assumption \ref{assum:RRM}. Equation~\eqref{eq:d^r-b^r>0} guarantees that the representative feature is recovered because the persistence of the feature is greater than zero.
\end{proof}

\begin{remark}
    One advantage of the Representative Rips Merge method is that it allows for some unrecovered sub-regions, so the union bound in Lemma \ref{lem:sub-features can be recovered} may be relaxed.

    According to Equation~\eqref{eq:l1bound} it is not expected that every sub-region has enough data points to recover its corresponding sub-feature under Assumption~\ref{data-independent}. However, a bound on the number of unrecovered sub-regions is available. Per the discussion following Equation~\eqref{eq:l1bound}, with probability at least $1-\left(2^m-2\right) e^{-n  \wideeps^2 / 2}$, there are at most approximately $r=2 m\wideeps$ sub-regions that do not have enough points to recover a sub-feature. Therefore, Assumption~\ref{assum:RRM} can be modified as 
    \begin{align}
        \frac{\delta}{\cos\beta}+\frac{C\delta^2}{\cos^2\beta}+C \wideeps^{\frac{1}{\mo}}<\frac{d_\M-b_\M}{4},
    \end{align}
    and we can establish a similar result as Lemma~\ref{lem:merge-Representative Rips Merge method}. This allows the linearized computational complexity of the proposed DaC when $m=\Omega(n/\log n)$ because it removes the dependence of $m$ in the assumption. We reiterate that for the Representative Rips Merge method, $m$ has to be lower bounded by a constant depending on dimension $\mo$ so that the Rips filtration can identify a merged feature.

\end{remark}


\subsection{Consistency of DaC}
After the merge step successfully identifies which sub-regions are merged, we now derive theoretical guarantees for the DaC naive estimates, $\mbe$ and $\hd$, in the following theorem.

\begin{theorem}[Formal version of Theorem \ref{thm:informal}]\label{thm:bhat and b; dhat and d}
	Given Assumptions~\ref{data}, \ref{assum:manifold angle}, \ref{assum:error bound}, and \ref{assum:error bound for delta: individual M}, with probability at least $ 1-\sum_{l=1}^m \frac{a_{\M^{(l)}}}{\sigma_k(\frac{\delta}{6})^{\mo}}\exp\left(-\frac{\sigma_k n c_{\min} c_\phi(\frac{\delta}{2})^{\mo}}{ a_{\M^{(l)}} }\right)$,
	\begin{equation}\label{eq:estimate bound for b and d}
	b_\M \le \mbe \le b_\M+4\delta ; \quad |d_\M-\hd|\le\frac{4\delta}{\cos\beta}+\frac{C \delta^2}{\cos^2\beta},
	\end{equation} 
 and 
 \begin{align}\label{eq:btrue and dtrue}
      b_\M \le \true{b} \le b_\M+2\delta; \quad | d_\M-\true{d}|\le 2\delta,
 \end{align}
and 
\begin{align}\label{eq:btrue<hatb<btrue+delta;dhat-dtrue<delta}
    \true{b}\le \hat{b}\le \true{b}+4\delta; \quad |\hat{d}-\true{d}|\le \frac{4\delta}{\cos\beta}+\frac{C \delta^2}{\cos^2\beta},
\end{align}
where constant $C>1$ only depends on $\phi$ and intrinsic properties of $\M$, but does not depend on the data. 
\end{theorem}

\begin{proof}
    Recall that $b_\M=0$, therefore Equation~\eqref{eq:btrue and dtrue} is a result of Proposition~\ref{prop:manifold}. In the following, we first prove Equation~\eqref{eq:estimate bound for b and d}, followed by Equation~\eqref{eq:btrue<hatb<btrue+delta;dhat-dtrue<delta}.
\begin{enumerate}
    \item \textbf{(Birth time)} $\hat{b}\ge b_\M$ because $b_\M = 0$. On the other hand, the birth time estimate $\hb$ can either be in a single sub-region or across two nearby sub-regions. By applying Boole's inequality to $\M^{(l)}$ as in Lemma~\ref{lem:sub-features can be recovered}, with probability at least $ 1-\sum_{l=1}^m \frac{a_{\M^{(l)}}}{\sigma_k(\frac{\delta}{6})^{\mo}}\exp\left(-\frac{\sigma_k nc_{\min}c_\phi(\frac{\delta}{2})^{\mo}}{a_{\M^{(l)}} }\right) $, Equation~\eqref{eq:b<delta<d_M<d} holds for any $l\in[m]$. Applying the triangle inequality, 
    \begin{align}
        \hb\le \hb^i+\hb^j\le b_\M+4\delta,
    \end{align}
    where $\hb^i$ and $\hb^j$ are the birth estimates for sub-regions $i$ and $j$ that determines the birth time, $\hb$. The inequality of birth time in Equation~\eqref{eq:estimate bound for b and d} follows from Equation~\eqref{eq:b<delta<d_M<d}, as we find an $\delta$-covering in every $\M^{(l)}$ such that there exists at least one $y_j\in \Yn$ in every $\delta$-ball.
    \item  \textbf{(Death time)} While the birth time considers the Euclidean distance between nearby points, the estimate of death time considers the Euclidean distance between points globally. We quantify the range of missing data as in Lemma~\ref{lem:masked points on Mi}. First, we consider a union bound on Lemma~\ref{lem:sub-features can be recovered} for any $\M^{(l)}$. Given any two points $y_i^\M$ and $y_j^\M$ from $\M$, with their closest unmasked data points from $\Yn$ denoted $y_i$ and $y_j$, respectively, then the following holds:
    \begin{equation}
        \begin{split}
            \rho(y_i,y_j)
        &\le \rho(y_i,y_i^\M)+\rho(y_i^\M,y_j^\M)+\rho(y_j^\M,y_j)\\
        &\le \rho(y_i^\M,y_j^\M)+\frac{4\delta}{\cos\beta}+\frac{C \delta^2}{\cos^2\beta},
        \end{split}
    \end{equation}
    where the last inequality is due to the masking effect on $\B$ being at most $\frac{2\delta}{\cos\beta}+\frac{C \delta^2}{\cos^2\beta}$. This guarantees that $|\hat{d}-d_\M|\le \frac{4\delta}{\cos\beta} + \frac{C \delta^2}{\cos^2\beta}$.
\end{enumerate}

    Finally, we prove Equation~\eqref{eq:btrue<hatb<btrue+delta;dhat-dtrue<delta}. Notice $\hb$ is lower bounded by $\true{b}$ because the recovered data is a subset of $Y_{1:n}$. For $\hat{b}\le \true{b}+4 \delta$, we follow the same proof as for the birth time in Equation~\eqref{eq:estimate bound for b and d} by applying the union bound of Lemma~\ref{lem:sub-features can be recovered} on $\M^{(l)}$ for arbitrary $l$. The $|\hd-\true{d}|$ bound follows similarly.
\end{proof}

This result immediately guarantees an upper bound on the $l_\infty$ distance as follows,
\begin{align}
    \left\|\binom{\hb}{\hd}-\binom{\true{b}}{\true{d}}\right\|_\infty \le \frac{4\delta}{\cos\beta}+\frac{C \delta^2}{\cos^2\beta},
\end{align}
which is generalized to the bottleneck distance in the next section when considering multiple manifolds.

\begin{remark}
    Theorem~\ref{thm:bhat and b; dhat and d} guarantees with high probability that there are enough representative points recovered using the proposed DaC method to estimate the birth and death times of the feature when $m_{\min} \le m\le n/\log n $, where $m_{\min}$ is a constant that depends on $\M$.
    This lower bound of $m$ is to ensure there are enough sub-regions for the merge step of the Representative Rips Merge method to identify a feature. 
    Assuming the sub-regions are equal-sized and $c_{\min}\sim 1/m$, then $a_{\M^{(l)}}\sim 1/m$, and $\min_{l}\left(\frac{d_{\M^{(l)}\cup\bar{\B}^{(l)}}-b_{\M^{(l)}\cup\bar{\B}^{(l)}} }{4}\right)\sim 1/(m^{1/\mo})$. The probability in Theorem~\ref{thm:bhat and b; dhat and d} reduces to $1-C\delta^{-\mo}\exp\left(-{c n  c_\phi\delta^{\mo}}\right)$. When $m\sim n/\log n$ and $\delta\sim \min_{l}\left(\frac{d_{\M^{(l)}\cup\bar{\B}^{(l)}}-b_{\M^{(l)}\cup\bar{\B}^{(l)}} }{4}\right)\sim 1/(m^{1/k})$, then the probability reduces to $1-1/\log n\to 1$ when $n\to\infty$. This shows that the proposed DaC allows linearized computational complexity and memory cost.
\end{remark}

\begin{remark}
The angle assumption, $ \cos\beta $, affects the bound of Lemma~\ref{lem:masked points on Mi}. If $ \beta $ is close to $\pi/2$, the resulting proximity of $\B$ and $\M$ may result in points being masked due to the boundary. Consequently, the local topological information is distorted, which affects the death time estimate.
One potential approach to improve the union bound on $\beta$ is to adapt the orientation of the supplemental boundary so it is approximately orthogonal to $\M$ (e.g., using local PCA on the data points near $\B\cap\M$ to find the orthogonal direction and compute the distance matrix). This could make $\beta\approx 0$. On the other hand, we observe from empirical experiments that $\beta$ does not significantly affect the DaC birth and death time estimates, even when $m$ is very large. 

\end{remark}


\section{Multiple manifolds} \label{sec:multiple}
In this section, we discuss the extension of the result from Theorem~\ref{thm:bhat and b; dhat and d} to multiple sufficiently separated features. Let $Y_{1:n} = \{y_i\}_{i=1}^n \in \R^{D}$ be data sampled independently from distribution $\mu$ supported on a set $\M$ of the form
\begin{align}
    \M= \M_1\cup\M_2\cup\dots\cup \M_K,
\end{align}
where each $\M_l$, $l=1,\ldots, K$, is a smooth, compact, connected manifold without boundary with positive curvature almost everywhere and is isomorphic to a $\mo$-sphere.  To avoid superfluous features, the $\M_l$'s must be spaced so that the following holds, 
\begin{align}
    \min_{l\not= l'}\dist(\M_l,\M_{l'})\ge \kappa,
\end{align}
where $\kappa>0$ represents the minimum distance among different manifolds. The distribution $\mu$ is assumed to be a $K$-component mixture model taking the form
	\begin{align}
	    d \mu= w_1\phi_1 d\vol_{\M_1}   + \dots +w_K \phi_K d \vol_{\M_K},
	\end{align}
	for smooth density functions $\phi_l: \M_l \rightarrow \R$ and positive weights $w_l$ such that $\sum_{l=1}^Kw_l = 1$.  We assume each $\phi_l$ satisfies Assumption \ref{assumption:probability}.

We denote $b^l_{\mathrm{True}}$ and $d^l_{\mathrm{True}}$ as the birth and death times using the full data on $\M_l$, and denote $b^l_\M$ and $d^l_\M$ to be the birth and death time of $\M_l$. Although the $\M_l$ are the features of the underlying manifold, extra features may exist when computing a persistent diagram using data with birth times larger than $\kappa$. 

Next we consider the two merge methods: the Projected Merge and Representative Merge methods. Different recovery results are guaranteed for each merge method. On the one hand, the Projected Merge method is not guaranteed to recover extra features in the data but only guaranteed to recover $\M_l$. Indeed, we can see the extra discrete features as geometric noise. On the other hand, the Representative Rips Merge method can recover the true features found using the full dataset. 

We follow the same proof procedure as in Section~\ref{appen:DaC}. We divide $\M$ into sub-regions $\M_l^{(i)}$, where $i=[m_l]$ and $m_l$ is the total number of sub-regions containing $\M_l$, such that $m\defeq \sum_{l=1}^K m_l$.

We assume $\hb^l$ and $\hd^l$ are the estimated DaC birth and death time corresponding to either $b^l_\M$ and $d^l_\M$ or $b^l_{\mathrm{True}}$ and $d^l_{\mathrm{True}}$ for feature $l$. This can also be viewed as an upper bound on the bottleneck distance between the DaC-based estimates and either $\M$ or the true estimates (using the full data).

Let $i_{\M_l}$, $S_{\M_l}$, and $R_{\M_l}$ be the injectivity radius, the upper bound of the absolute sectional curvature, and the reach, respectively, for $\M_l$.  Then
$i_\M = \min_l(i_{\M_l})$, $S_\M=\max_l(S_{\M_l})$, and $R_\M=\min_l(R_{\M_l})$. Before presenting our main result, we note extra conditions on $\delta$ as follows.

\begin{assumption}\label{assum:m and kappa}
The upper bound of the Euclidean diameter of the sub-regions ($C/(m^{1/\mo})$) satisfies
    \begin{equation}
        \frac{C}{m^{1/\mo}}< \kappa.
    \end{equation}
\end{assumption}
\noindent This assumption guarantees that the Euclidean diameter of the sub-region is small so that the sub-region cannot contain points from more than one feature $\M_l$.
\begin{assumption}\label{assum:delta for multimanifold}
The local error bound, $\delta>0$, satisfies
\begin{equation}
    0<\delta\le \min_{k,l}\left\{\frac{d_{\M^{(i)}_l\cup\bar{\B}^{(i)}_l}-b_{\M^{(i)}_l\cup\bar{\B}^{(i)}_l}}{4}\right\}.    
\end{equation}
\end{assumption}
\noindent This assumption resembles Assumption~\ref{assum:error bound for delta: individual M} but for the multi-manifold case.

\begin{theorem}[Projected Merge method]
    Given Assumptions~\ref{data}, \ref{assum:manifold angle}, \ref{assum:error bound}, \ref{assum:m and kappa}, and \ref{assum:delta for multimanifold} hold, with probability at least $ 1-\sum_{i}\sum_{l} \frac{a_{\M^{(i)}_l}}{c(\frac{\delta}{2})^{\mo}}\exp\left(-\frac{c n c_{\min} c_\phi(\frac{\delta}{2})^{\mo}}{ a_{\M^{(i)}_l} }\right) $, we have
	\begin{equation}
	   b_\M^l \le \mbe^l \le b_\M^l+4\delta ; \quad |d^l_\M-\hd^l|\le \frac{4\delta}{\cos\beta}+\frac{C \delta^2}{\cos^2\beta}.
	\end{equation} 
\end{theorem}

Proposition~\ref{prop:manifold} and Lemmas \ref{lem:sub-features can be recovered}, \ref{lem:masked points on Mi}, and \ref{lem:Projected Cancellation Method} follow from the assumptions. The only difference between this result and Theorem~\ref{thm:bhat and b; dhat and d} is that this result does not merge the sub-region from different manifolds because of Assumption~\ref{assum:m and kappa}. For the individual $\M_l$, the bound for the estimates directly follows from the proof in Theorem~\ref{thm:bhat and b; dhat and d}.

The Representative Rips Merge method is not guaranteed to recover $\M_l$, but it may recover extra features due to the proximity of multiple manifolds; these extra features would also appear on the true persistence diagram (using the full dataset). This method requires an alternative condition to guarantee the recovery of the $l^\mathrm{th}$ true feature such that $b^l_{\mathrm{True}}\ge\kappa$.

\begin{assumption}\label{assum:delta for representative merge}
Suppose that there are $r$ unrecovered sub-regions, then $\delta$, $m$, and $r$ must be such that
    $C\left(\frac{r}{m}\right)^{\frac{1}{\mo}}+\frac{\delta}{\cos\beta}+C\frac{\delta^2}{\cos^2\beta}\le \min_l\left\{\frac{d^l_{\mathrm{True}}-b^l_{\mathrm{True}}}{4},c_{\mathrm{count}}\right\}$ where constant $0<c_{\mathrm{count}}<1$ depends on intrinsic properties of $\M$ to ensure that the number of sub-features is large enough so that the Representative Rips Merge method can detect a representative homology group generator, and constant $C>1$ depends on $\phi$ and intrinstic properties of $\M$.
\end{assumption}
This assumption can be seen as an analog of Assumption \ref{assum:RRM} generalized to the multi-manifold setting.  By using the same reasoning as for Lemma \ref{lem:merge-Representative Rips Merge method}, recoverable sub-features can be merged using the Representative Rips Merge method.

\begin{theorem}[Representative Rips method]\label{thm:Multi-manifold:Representative Rips Method}
    Given Assumptions~\ref{data}, \ref{assum:manifold angle}, \ref{assum:error bound}, \ref{assum:m and kappa}, \ref{assum:delta for multimanifold}, and \ref{assum:delta for representative merge} hold, and the Representative Rips Merge method is used, when $n$ is large enough, with probability at least $ 1-\sum_{i}\sum_{l} \frac{a_{\M^{(i)}_l}}{c(\frac{\delta}{2})^{\mo}}\exp\left(-\frac{c n c_{\min} c_\phi(\frac{\delta}{2})^{\mo}}{ a_{\M^{(i)}_l} }\right)$, 
\begin{equation}
    |\true{b}^l- \hat{b}^l|\le 4\delta; \quad |\hat{d}^l-\true{d}^l|\le \frac{4\delta}{\cos\beta}+\frac{C \delta^2}{\cos^2\beta}.
\end{equation}
\end{theorem}
The proof is similar to the proof for Theorem~\ref{thm:bhat and b; dhat and d} applied to each $\M_l$, as Lemma \ref{lem:sub-features can be recovered}, \ref{lem:masked points on Mi}, and \ref{lem:merge-Representative Rips Merge method} hold.

\end{document}